\documentclass{article}
\usepackage{amssymb}
\usepackage{amsmath,amsthm}
\usepackage[latin1]{inputenc}
\usepackage[dvips]{graphicx}
\usepackage{hyperref}
\usepackage{color}
\usepackage{enumerate}
\usepackage{tikz}
\usepackage{xifthen}

\hypersetup{colorlinks=true}

\hypersetup{colorlinks=true, linkcolor=blue, citecolor=blue,urlcolor=blue}

\newtheorem{remark}{Remark}

\newtheorem{theorem}[remark]{Theorem}
\newtheorem{proposition}[remark]{Proposition}
\newtheorem{corollary}[remark]{Corollary}

\newtheorem{conjecture}[remark]{Conjecture}

\newcommand{\adim}{\operatorname{\mathrm{adim}}}
\newcommand{\Ext}{\operatorname{Ext}} 

\title{On the adjacency dimension of graphs}

\author{A. Estrada-Moreno, Y. Ram\'{\i}rez-Cruz, J. A. Rodr\'{\i}guez-Vel\'{a}zquez
\\
{\small Departament d'Enginyeria Inform\`{a}tica i Matem\`{a}tiques
}\\
{\small Universitat Rovira i Virgili,  Av. Pa\"{\i}sos Catalans 26, 43007
Tarragona, Spain.}
\\{\small e-mail:\mbox{\tt $\{$alejandro.estrada;yunior.ramirez;juanalberto.rodriguez$\}$\@@urv.cat} }
}
\date{}
\begin{document}
\maketitle

\begin{abstract}
A generator of a metric space is a set $S$ of points in the space with the property that every point
of the space is uniquely determined by its distances from the elements of $S$. Given a simple graph $G=(V,E)$, we define the distance function $d_{G,2}:V\times V\rightarrow \mathbb{N}\cup \{0\}$, as
$
d_{G,2}(x,y)=\min\{d_G(x,y),2\},  
$
where $d_G(x,y)$ is the length of a shortest path between $x$ and $y$ and $\mathbb{N}$ is the set of positive integers. Then $(V,d_{G,2 })$ is a metric space. We say that a set $S\subseteq V$ is a $k$-\textit{adjacency generator} for $G$ if for every two vertices $x,y\in V$, there exist at least $k$ vertices $w_1,w_2,...,w_k\in S$ such that
$$d_{G,2}(x,w_i)\ne d_{G,2}(y,w_i),\; \mbox{\rm for every}\; i\in \{1,...,k\}.$$  A minimum cardinality $k$-adjacency generator is called a $k$-\textit{adjacency basis} of $G$ and its cardinality, the $k$-\textit{adjacency dimension} of $G$. 

In this article we study the problem of finding the $k$-adjacency dimension of a graph. We give some necessary and sufficient conditions
for the existence of a $k$-adjacency basis of an arbitrary graph $G$ and we obtain general results on the $k$-adjacency dimension,  including  general bounds
and closed formulae for some  families of graphs.
In particular,  we obtain closed formulae  for the $k$-adjacency dimension of join graphs $G+H$ in terms of the $k$-adjacency dimension of $G$ and $H$. These results concern the  $k$-metric dimension, as join graphs have diameter two. 
As we can expect, the obtained results will become important tools for the study of the $k$-metric dimension of   lexicographic product  graphs and  corona product graphs. Moreover,  several results obtained in this paper need not be restricted to the metric $d_{G,2}$, they can be expressed in a more general setting, for instance, by using the metric $d_{G,t}(x,y)=\min\{d_G(x,y),t\}$  for  $t\in \mathbb{N}$. 
\end{abstract}

\section{Introduction}

A generator of a metric space $(X,d)$ is a set $S\subset X$ of points in the space  with the property that every point of $X$  is uniquely determined by the distances from the elements of $S$. Given a simple and connected graph $G=(V,E)$, we consider the function $d_G:V\times V\rightarrow \mathbb{N}\cup \{0\}$, where $d_G(x,y)$ is the length of a shortest path between $u$ and $v$ and $\mathbb{N}$ is the set of positive integers. Then $(V,d_G)$ is a metric space since $d_G$ satisfies $(i)$ $d_G(x,x)=0$  for all $x\in V$,$(ii)$  $d_G(x,y)=d_G(y,x)$  for all $x,y \in V$ and $(iii)$ $d_G(x,y)\le d_G(x,z)+d_G(z,y)$  for all $x,y,z\in V$. A vertex $v\in V$ is said to \textit{distinguish} two vertices $x$ and $y$ if $d_G(v,x)\ne d_G(v,y)$.
A set $S\subset V$ is said to be a \emph{metric generator} for $G$ if any pair of vertices of $G$ is
distinguished by some element of $S$. A minimum cardinality metric generator is called a \emph{metric basis}, and
its cardinality the \emph{metric dimension} of $G$, denoted by $\dim(G)$.


The notion of metric dimension of a graph was introduced by Slater in \cite{Slater1975}, where the metric generators were called \emph{locating sets}. Harary and Melter independently introduced the same concept in  \cite{Harary1976}, where metric generators were called \emph{resolving sets}. Applications of this invariant to the navigation of robots in networks are discussed in \cite{Khuller1996} and applications to chemistry in \cite{Johnson1993,Johnson1998}.  Several variations of metric generators, including resolving dominating sets \cite{Brigham2003}, independent resolving sets \cite{Chartrand2003}, local metric sets \cite{Okamoto2010}, strong resolving sets \cite{Sebo2004}, adjacency resolving sets  \cite{JanOmo2012}, $k$-metric generators \cite{Estrada-Moreno2013,Estrada-Moreno2013corona}, etc., have since been introduced and studied. In this article, we focus on the last of these issues: we are interested in the study of adjacency  resolving sets and $k$-metric generators.

The concept of adjacency generator\footnote{Adjacency generators were called adjacency resolving sets in   \cite{JanOmo2012}} was introduced by Jannesari and Omoomi in \cite{JanOmo2012} as a tool to study the metric dimension of lexicographic product graphs. This concept has been studied further by Fernau and Rodr\'{i}guez-Vel\'{a}zquez in \cite{Rodriguez-Velazquez-Fernau2013,Fernau-Ja-Corona-2014} where they showed
that the (local) metric dimension of the corona product of a graph of order $n$ and some
non-trivial graph $H$ equals $n$ times the (local) adjacency  dimension of $H$. As a consequence of this strong relation they showed that the problem of computing the adjacency dimension is
NP-hard.   A set $S\subset V$ of vertices in a graph $G=(V,E)$ is said to be  an \emph{adjacency generator} for $G$  if for every two vertices $x,y\in V\setminus S$ there exists $s\in S$ such that $s$ is adjacent to exactly one of $x$ and $y$. A minimum cardinality adjacency generator is called an \emph{adjacency basis} of $G$, and its cardinality  the \emph{adjacency dimension} of $G$,  denoted by $\adim(G)$.

Notice that $S$ is an adjacency generator for $G$ if and only if $S$ is an adjacency generator for its complement $\overline{G}$. This is justified by the fact that given an adjacency generator $S$ for $G$, it holds that for every $x,y\in V\setminus S$ there exists $s\in S$ such that $s$ is adjacent to exactly one of $x$ and $y$, and this property holds in $\overline{G}$. Thus, $\adim(G)=\adim(\overline{G}).$ Besides, from the definition of adjacency and metric bases, we deduce that $S$ is an adjacency basis of a graph $G$ of diameter at most two if and only if $S$ is a metric basis of $G$. In these cases, $\adim(G)=\dim(G)$.


As pointed out in \cite{Rodriguez-Velazquez-Fernau2013,Fernau-Ja-Corona-2014}, 
any  adjacency generator of a graph $G=(V,E)$ is  also a metric generator in a suitably chosen metric space.
Given a positive integer $t$,  we define the distance function $d_{G,t}:V\times V\rightarrow \mathbb{N}\cup \{0\}$, where
\begin{equation*}\label{distinguishAdj}
d_{G,t}(x,y)=\min\{d_G(x,y),t\}.
\end{equation*}
Then any metric generator for $(V,d_{G,t})$ is a metric generator for $(V,d_{G,t+1})$ and, as a consequence, the metric dimension of $(V,d_{G,t+1})$  is less than or equal to the metric dimension of $(V,d_{G,t})$. In particular, the metric dimension of $(V,d_{G,1})$ is equal to $|V|-1$,  the metric dimension of $(V,d_{G,2})$ is equal to $\adim(G)$ and, if $G$ has diameter $D(G)$, then $d_{G,D(G)}=d_G$ and so  the metric dimension of  $(V,d_{G,D(G)})$  is equal to $\dim(G)$.
Notice that when using the metric $d_{G,t}$ 
the concept of metric generator needs not be restricted to the case of connected graphs\footnote{For any pair of vertices $x,y$ belonging to different connected components of $G$ we can assume that $d_G(x,y)=\infty>2$ and so $d_{G,t}(x,y)=t$.}.

The concept of $k$-metric generator  introduced by Estrada-Moreno, Yero and Rodr\'{i}guez-Vel\'{a}zquez  \cite{Estrada-Moreno2013corona,Estrada-Moreno2014}, is a natural extension of the concept of metric generator.
A set $S\subseteq V$ is said to be a \emph{$k$-metric generator} for $G$ if and only if any pair of vertices of $G$ is distinguished by at least $k$ elements of $S$, {\em i.e.}, for any pair of different vertices $u,v\in V$, there exist at least $k$ vertices $w_1,w_2,...,w_k\in S$ such that $$d_G(u,w_i)\ne d_G(v,w_i),\; \mbox{\rm for every}\; i\in \{1,...,k\}.$$ A $k$-metric generator of minimum cardinality in $G$ is called a \emph{$k$-metric basis}, and its cardinality the $k$-metric dimension of $G$,  denoted by $\dim_{k}(G)$.

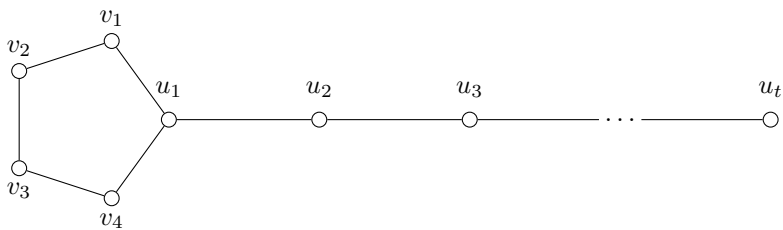
\begin{figure}[!ht]
\centering
\begin{tikzpicture}[transform shape, inner sep = .7mm]
\def\radius{1.1} 
\foreach \ind in {1,...,4}
{
\pgfmathparse{360/5*\ind};
\node [draw=black, shape=circle, fill=white] (v\ind) at (\pgfmathresult:\radius cm) {};
\ifthenelse{\ind=3\OR \ind=4}
{
\node [scale=1] at ([yshift=-.3 cm]v\ind) {$v_\ind$};
}
{
\node [scale=1] at ([yshift=.3 cm]v\ind) {$v_\ind$};
};
}
\foreach \ind in {1,...,3}
\pgfmathparse{int(\ind+1)}
\draw[black] (v\ind) -- (v\pgfmathresult);
\foreach \ind in {1,...,3}
{
\pgfmathparse{\radius + 2*(\ind-1)};
\node [draw=black, shape=circle, fill=white] (u\ind) at (\pgfmathresult cm, 0) {};
\node [scale=1] at ([yshift=.4 cm]u\ind) {$u_\ind$};
}
\pgfmathparse{\radius + 2*3};
\node (ldot) at (\pgfmathresult cm, 0) {$\ldots$};
\pgfmathparse{\radius + 2*4};
\node [draw=black, shape=circle, fill=white] (ut) at (\pgfmathresult cm, 0) {};
\node [scale=1] at ([yshift=.4 cm]ut) {$u_t$};
\foreach \ind in {1,4}
\draw[black] (u1) -- (v\ind);
\foreach \ind in {1,2}
\pgfmathparse{int(\ind+1)}
\draw[black] (u\ind) -- (u\pgfmathresult);
\draw[black] (u3) -- (ldot);
\draw[black] (ldot) -- (ut);
\end{tikzpicture}
\caption{\label{figDimk}For $k\in \{1,2,3,4\}$, $\dim_k(G)=k+1$.}
\end{figure}

As an  example we take a graph $G$ obtained from the cycle graph $C_5$ and the path $P_t$, by identifying one of the vertices of the cycle, say $u_1$, and one of the extremes of $P_t$, as we show in Figure \ref{figDimk}. Let $S_1=\{v_1,v_2\}$, $S_2=\{v_1,v_2,u_t\}$, $S_3=\{v_1,v_2,v_3,u_t\}$ and $S_4=\{v_1,v_2,v_3,v_4,u_t\}$. For $k\in \{1,2,3,4\}$ the set $S_k$ is $k$-metric basis of $G$.

Note that every $k$-metric generator $S$ satisfies that $|S|\geq k$ and, if $k>1$, then $S$ is also a $(k-1)$-metric generator. Moreover, $1$-metric generators are the standard metric generators (resolving sets or locating sets as defined in \cite{Harary1976} or \cite{Slater1975}, respectively). Some basic results on the $k$-metric dimension of a graph  have recently been obtained in \cite{Estrada-Moreno2013,Estrada-Moreno2013corona,Estrada-Moreno2014b,Estrada-Moreno2014,Yero2013c}. In particular, it was shown in \cite{Yero2013c} that the problem of computing the $k$-metric dimension of a graph  is NP-hard.



We say that a set $S\subseteq V(G)$ is a $k$-\textit{adjacency generator} for $G$ if for every two vertices $x,y\in V(G)$, there exist at least $k$ vertices $w_1,w_2,...,w_k\in S$ such that
$$d_{G,2}(x,w_i)\ne d_{G,2}(y,w_i),\; \mbox{\rm for every}\; i\in \{1,...,k\}.$$  A minimum $k$-adjacency generator is called a $k$-\textit{adjacency basis} of $G$ and its cardinality, the $k$-\textit{adjacency dimension} of $G$, is denoted by $\adim_k(G)$. For connected graphs, any $k$-adjacency basis is a $k$-metric   basis. Hence, if there exists a
$k$-adjacency basis of a connected graph $G$, then $$\dim_k(G)\le \adim_k(G).$$ Moreover, if $G$ has diameter at most two, then $\dim_k(G)= \adim_k(G).$

For the graph $G$ shown in Figure \ref{figExampleOfDifference} we have $\dim_1(G)=8<9=\adim_1(G)$, $\dim_2(G)=12<14=\adim_2(G)$ and $\dim_3(G)=20=\adim_3(G)$. Note that the only $3$-adjacency basis of $G$, and at the same time the only $3$-metric basis, is $V(G)-\{0,6,{12},{18}\}$.

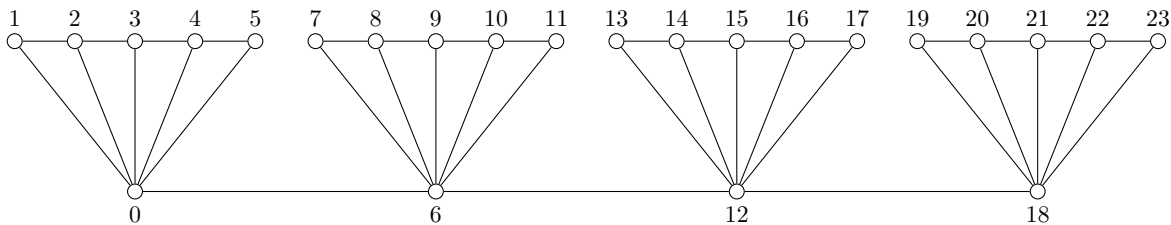
\begin{figure}[!ht]
\centering
\begin{tikzpicture}[transform shape, inner sep = .7mm]
\foreach \x in {1,...,4}
{
\pgfmathtruncatemacro{\iind}{6*(\x-1)};
\pgfmathsetmacro{\xc}{(5*(\x-1)*.8};
\node [draw=black, shape=circle, fill=white] (v\iind) at (\xc,0) {};
\node [scale=.9] at ([yshift=-.3 cm]v\iind) {${\iind}$};
\foreach \ind in {1,...,5}
{
\pgfmathsetmacro{\xc}{(5*(\x-1)+\ind-3)*.8};
\pgfmathparse{int(\ind + 6*(\x-1))};
\node [draw=black, shape=circle, fill=white] (v\pgfmathresult) at (\xc,2) {};
\node [scale=.9] at ([yshift=.3 cm]v\pgfmathresult) {${\pgfmathresult}$};
\draw[black] (v\iind) -- (v\pgfmathresult);
\ifthenelse{\ind>1}
{
\pgfmathtruncatemacro{\bind}{\ind + 6*(\x-1)-1};
\draw[black] (v\pgfmathresult) -- (v\bind);
}
{};
}

\ifthenelse{\x>1}
{
\pgfmathparse{6*(\x-2)};
\draw[black] (v\iind) -- (v\pgfmathresult);
}
{};
}

\end{tikzpicture}
\caption{\label{figExampleOfDifference}\label{figExampleOfDifference}The set  $\{2,4,6,8,{10},{14},{16},{20},{21}\}$ is an adjacency basis of $G$, while the set $\{{2l+1}:\; l\in \{0,...,11\}\}\cup\{6,{12}\}$ is a $2$-adjacency  basis and $V(G)-\{0,6,{12},{18}\}$ is a $3$-adjacency  basis.}
\end{figure}


In this article we study the problem of finding the $k$-adjacency dimension of a graph. The paper is organized as follows: in Section \ref{SectionDimensional} we give some necessary and sufficient conditions
for the existence of a $k$-adjacency basis of an arbitrary graph $G$, \textit{i.e.,} we
determine the range of $k$  where $\adim_k(G)$ makes sense.
Section \ref{secAdjDim} is devoted to the study of the $k$-adjacency dimension. We obtain general results on this invariants  including  tight bounds
and closed formulae for some particular families of graphs.
Finally, in Section \ref{SectionJoin}  we obtain closed formulae  for the $k$-adjacency dimension of join graphs $G+H$ in terms of the $k$-adjacency dimension of $G$ and $H$. These results concern the  $k$-metric dimension, as join graphs have diameter two. 

As we can expect, the obtained results will become important tools for the study of the $k$-metric dimension of   lexicographic product  graphs and  corona product graphs. Moreover, we would point out that several results obtained in this article, like those in Remark \ref{Remark-Kmetric-adjmetric} 
and subsequent, until  Theorem \ref{theoAdjDimkn}, need not be restricted to the metric $d_{G,2}$, they can be expressed in a more general setting, for instance, by using the metric $d_{G,t}$  for any positive integer $t$. 

We will use the notation $K_n$, $K_{r,s}$, $C_n$, $N_n$ and $P_n$ for complete graphs,  complete bipartite graphs, cycle graphs, empty graphs and path graphs, respectively. 
We use the notation $u\sim v$ if $u$ and $v$ are adjacent and $G\cong H$ if  $G$ and $H$ are isomorphic graphs. For a vertex $v$ of a graph $G$, $N_G(v)$ will denote the set of neighbours or \emph{open neighborhood} of $v$ in $G$, \textit{i.e.}, $N_G(v)=\{u\in V(G):\; u\sim v\}$. The \emph{closed neighborhood}, denoted by $N_{G}[x]$, equals $N_{G}(x)\cup \{x\}$. If there is no  ambiguity, we will simply write  $N(x)$ or $N[x]$.
We also define $\delta(v)=|N(v)|$ as the degree of vertex $v$, as well as, $\delta(G)=\min_{v\in V(G)}\{\delta(v)\}$ and $\Delta(G)=\max_{v\in V(G)}\{\delta(v)\}$.  The subgraph induced by a set $S$ of vertices will be denoted  by $\langle S\rangle$, the diameter of a  graph will be denoted by $D(G)$ and the girth by $\mathtt{g}(G)$.
For the remainder of the paper, definitions will be introduced whenever a concept is needed.

\section{$k$-adjacency dimensional graphs}  \label{SectionDimensional}

We say that a graph $G$ is  \emph{$k$-adjacency dimensional} if $k$ is the largest integer such that there exists a $k$-adjacency basis of $G$. Notice that if $G$ is a $k$-adjacency dimensional graph, then for each positive integer $r\le k$, there exists at least one $r$-adjacency basis of $G$.
Given a connected graph $G$ and  two different vertices $x,y\in V(G)$, we denote by $\mathcal{C}_G(x,y)$ the  set of vertices that distinguish  the pair $x,y$ with regard to the metric $d_{G,2}$, \textit{ i.e.}, $$\mathcal{C}_G(x,y)=\{z\in V(G):\; d_{G,2}(x,z)\ne d_{G,2}(y,z)\}.$$
Then a set  $S\subseteq V(G)$ is a $k$-adjacency generator for $G$ if $|\mathcal{C}_G(x,y)\cap S|\ge k$ for all $x,y\in V(G)$.  Notice that two vertices $x,y$ are twins if and only if
$\mathcal{C}_G(x,y)=\{x,y\}.$

Since for every $x,y\in V(G)$ we have that $|\mathcal{C}_G(x,y)|\ge 2$, it follows that the whole vertex set $V(G)$ is a $2$-adjacency generator for $G$ and, as a consequence, we deduce that every graph $G$ is $k$-adjacency dimensional for some $k\ge 2$. On the other hand, for any  graph $G$ of order $n\ge 3$, there exists at least one vertex $v\in V(G)$ such that $|N_G(v)|\ge 2$ or $|V(G)-N_G(v)|\ge 2$, so for any pair $x,y\in N_G(v)$ or $x,y\in V(G)-N_G(v)$, we deduce that $v\notin \mathcal{C}_G(x,y)$ and, as a result,  there is no $n$-adjacency dimensional graph of order $n\ge 3$.

We define the following parameter $$\mathcal{C}(G)=\displaystyle\min_{x,y\in V(G)}\{|\mathcal{C}_G(x,y)|\}.$$

\begin{theorem}\label{theokadjacency}
A graph $G$ is $k$-adjacency dimensional if and only if $k=\mathcal{C}(G)$. Moreover,  $\mathcal{C}(G)$ can be computed in $O(|V(G)|^3)$ time.
\end{theorem}

\begin{proof} First we shall prove the equivalence.
(Necessity) If $G$ is a $k$-adjacency dimensional graph, then for any $k$-adjacency basis $B$ and any pair of vertices $x,y\in V(G)$, we have
$|B\cap {\cal C}_G(x,y)|\ge k$. Thus, $k\le \mathcal{C}(G)$. Now we suppose that $k<\mathcal{C}(G)$. In such a case, for every $x',y'\in V(G)$ such that $|B\cap {\cal C}_G(x',y')|=k$, there exists $z_{x'y'}\in {\cal C}_G(x',y')-B$ such that $d_{G,2}(z_{x'y'},x')\ne d_{G,2}(z_{x'y'},y')$. Hence, the set $$B\cup \left(\displaystyle\bigcup_{x',y'\in V(G):\; |B\cap {\cal C}_G(x',y')|=k}\{z_{x'y'}\}\right)$$ is a $(k+1)$-adjacency generator for $G$, which is a contradiction. Therefore, $k=\mathcal{C}(G).$

(Sufficiency) Let $a,b\in V(G)$ such that $\displaystyle\min_{x,y\in V(G)}|{\cal C}_G(x,y)|=|{\cal C}_G(a,b)|=k$. Since no set $S\subseteq V(G)$ satisfies $|S\cap {\cal C}_G(a,b)|>k$ and $V(G)$ is a $k$-adjacency generator for $G$,  we conclude that $G$ is a $k$-adjacency dimensional graph.

Now, we assume that the graph $G$ is represented by its adjacency matrix ${\bf A}$. We recall that ${\bf A}$ is a symmetric $(n\times n)$-matrix given by
$${\bf A}(i,j)=\left\{\begin{array}{ll}
                                1, & \mbox{if $u_i\sim u_j$}, \\
                                0, & \mbox{otherwise.}
                              \end{array}\right.$$

Now observe that for every $z\in V(G)-\{x,y\}$ we have that $z\in {\cal C}_G(x,y)$ if and only if ${\bf A}(x,z)\ne {\bf A}(y,z)$. Considering this, we can compute $|{\cal C}_G(x,y)|$ in linear time for each pair $x,y\in V(G)$. Therefore, the overall running time for determining $\mathcal{C}(G)$ is dominated by the cubic time of computing the value of $|{\cal C}_G(x,y)|$ for $\displaystyle\binom{|V(G)|}{2}$ pairs of vertices $x,y$ of $G$.
\end{proof}

As Theorem \ref{theokadjacency} shows, given a graph $G$ and a positive integer $k$, the problem of deciding if $G$ is $k$-adjacency dimensional is easy to solve. Even so, we would point out some useful particular cases.  

\begin{remark}\label{coro2Addimesional}
A graph $G$ is $2$-adjacency dimensional if and only if there are at least two vertices of $G$ belonging to the same twin equivalence class.
\end{remark}

Note that by the previous remark we deduce that graphs such as the complete graph $K_n$ and the complete bipartite graph $K_{r,s}$ are $2$-adjacency dimensional. 

If $u,v\in V(G)$ are adjacent vertices of degree two and they are not twin vertices, then $|{\cal C}_G(u,v)|=4$. Thus, For any integer $n\ge 5$, $C_n$ is $4$-adjacency dimensional and  we can state the following more general remark.

\begin{remark}Let $G$ be a twins-free\footnote{A graph is twins-free  if all its twin equivalence classes are singleton.} graph of minimum degree two. If $G$ has two adjacent vertices of degree two, then $G$ is $4$-adjacency dimensional. 
\end{remark}


For any hypercube $Q_r$, $r\ge 2$, we have $|\mathcal{C}_{Q_r}(u,v)|=2r$ if $u\sim v$,  $|\mathcal{C}_{Q_r}(u,v)|=2r-2$ if $d_{Q_r}(u,v)=2$  and  $|\mathcal{C}_{Q_r}(u,v)|=2r+2$ if $d_{Q_r}(u,v)\ge 3$. Hence,  $\mathcal{C}(Q_r)=2r-2.$

\begin{remark}
For any integer $r\ge 2$ the hypercube $Q_r$ is $(2r-2)$-adjacency dimensional.
\end{remark}

It is straightforward that for any graph $G$ of girth $\mathtt{g}(G)\ge 5$ and minimum degree $\delta(G)\ge 2$,   $\mathcal{C}(G)\ge 2\delta(G).$ Hence, the following remark is immediate. 

\begin{remark}
Let $G$ be a $k$-adjacency dimensional graph. If  $\mathtt{g}(G)\ge 5$  and $\delta(G)\ge 2$, then $k\ge 2\delta(G)$.
\end{remark}

If there is an end-vertex\footnote{An end-vertex  of a graph $G$ is a vertex of degree one and its support vertex is its neighbour.} $u$ in $G$ whose support vertex $v$  has degree two, then $|\mathcal{C}_G(u,v)|=|N_G[v]|=3$. 
Hence,  we deduce the following result.

\begin{remark}\label{4-adjDim3-adjDim}
Let $G$ be a twins-free graph. If there exists an end-vertex whose support vertex has degree two, then $G$ is $3$-adjacency dimensional. 
\end{remark}

The case of trees is summarized in the following remark. Before stating it, we need some additional terminology.  Let $T$ be a tree. A vertex of degree at least 3 is called a \textit{major vertex} of  $T$.
A leaf $u$ of $T$ is said to be a {\em terminal vertex of a major vertex} $v$ of $T$ if $d_T(u,v)<d_T(u,w)$ for every other major vertex $w$ of $T$.
The {\em terminal degree} of a major vertex $v$ is the number of terminal vertices of $v$.
A major vertex $v$ of $T$ is an \textit{exterior major vertex} of $T$ if it has positive terminal degree.

\begin{remark}
Let $T$ be a $k$-adjacency dimensional tree of order $n\ge 3$. Then $k\in \{2,3\}$ and  $k=2$ if and only if there are two leaves sharing a common support vertex.
\end{remark}

\begin{proof}
By Remark \ref{coro2Addimesional} we conclude that $k=2$ if and only if there are two leaves sharing a common support vertex.
Also, if $T$ is a path different from $P_3$, then by Remark \ref{4-adjDim3-adjDim} we have that $k=3$. 

If $T$ is not a path, then there exists at least one exterior major vertex $u$ of terminal degree greater than one. Then,  either $u$ is the support vertex of all its terminal vertices, in which case Remark \ref{coro2Addimesional}  leads to $k=2$, or $u$ has at least one terminal vertex whose support vertex has degree two,  in which case Remark \ref{4-adjDim3-adjDim} leads to $k=3$ if there are no leaves of $T$  sharing a common support vertex.
\end{proof}

Since   $|\mathcal{C}_G(x,y)|\le \delta(x)+\delta(y)+2$, for all $x,y\in V(G)$, the following remark immediately follows.

\begin{remark}
If $G$ is a $k$-adjacency dimensional graph, then $$k\le \min_{x,y\in V(G)}\{\delta(x)+\delta(y)\}+2.$$
\end{remark}

This bound is achieved, for instance, for any  graph $G$ constructed as follows. Take a cycle $C_n$ whose vertex set is $V(C_n)=\{u_1,u_2,...,u_n\}$ and an empty graph $N_n$ whose vertex set is $V(N_n)=\{v_1,v_2,...,v_n\}$ and then, for $i=1$ to $n$, connect by an edge $u_i$ to $v_i$. In this case,   $G$ is $4$-adjacency dimensional. Also, a trivial  example is the case of graphs having two isolated vertices, which are $2$-adjacency dimensional.

As  defined in \cite{Estrada-Moreno2013}, a connected graph $G$ is  \emph{$k$-metric dimensional} if $k$ is the largest integer such that there exists a $k$-metric basis.
Since any $k$-adjacency generator is a $k$-metric generator, the following result is straightforward.

\begin{remark}\label{Remark-Kmetric-adjmetric}
If a graph $G$ is $k$-adjacency dimensional  and $k'$-metric dimensional, then $k\le k'$. Moreover, if $D(G)\le 2$, then $k'=k$.
\end{remark}

\section{$k$-adjacency dimension. Basic results}\label{secAdjDim}

In this section we present some results that allow us to compute the $k$-adjacency dimension of several families of graphs. We also give some tight bounds on the $k$-adjacency dimension of a graph.

\begin{theorem}[Monotony]\label{theoMonotyForAdjKs}
Let $G$ be a $k$-adjacency dimensional graph and let $k_1,k_2$ be two integers. If $1\le k_1<k_2\le k$, then $\adim_{k_1}(G)<\adim_{k_2}(G)$.
\end{theorem}

\begin{proof}
Let $B$ be a $k$-adjacency basis of $G$. Let $x\in B$. Since $|B\cap\mathcal{C}_G(y,z)|\ge k$, for all $y,z\in V(G)$,   we have that $B-\{x\}$ is a $(k-1)$-adjacency generator for $G$ and, as a consequence, $\adim_{k-1}(G)\le \left|B-\{x\}\right|<|B|=\adim_{k}(G)$. By analogy we deduce that $\adim_{k-2}(G)<\adim_{k-1}(G)$ and, repeating this process until  we get $\adim(G)<\adim_{2}(G)$, we obtain  the result.    
\end{proof}

\begin{corollary}\label{firstConsequenceAdjMonotony}
Let $G$ be a $k$-adjacency dimensional graph of order $n$.
\begin{enumerate}[{\rm (i)}]
\item For any $r\in\{2,..., k\}$, $\adim_r(G)\ge \adim_{r-1}(G)+1.$
\item For any $r\in\{1,..., k\}$, $\adim_r(G)\ge \adim(G)+(r-1).$
\item For any $r\in\{1,..., k-1\}$,  $\adim_r(G)<n$.
\end{enumerate}
\end{corollary}

For instance, for the Petersen graph we have $\adim_6(G)=\adim_5(G)+1=\adim_4(G)+2=\adim_3(G)+3=10$ and $\adim_2(G)=\adim_1(G)+1=4.$

In order to continue presenting our results, we need to define a new parameter:  $${\cal C}_k(G)=\bigcup_{|{\cal C}_G(x,y)|=k}{\cal C}_G(x,y).$$

For any $k$-adjacency basis $A$ of a $k$-adjacency dimensional graph $G$, it holds that every pair of vertices $x,y\in V(G)$ satisfies $|A\cap\mathcal{C}_{G}(x,y)|\ge k$. Thus, for every $x,y\in V(G)$ such that $|{\cal C}_G(x,y)|=k$ we have that ${\cal C}_G(x,y)\subseteq A$, and so ${\cal C}_k(G)\subseteq A$. The following result is a direct consequence of this.

\begin{remark}\label{remAdjTauk}
If $G$ is a $k$-adjacency dimensional graph and $A$ is a $k$-adjacency basis, then ${\cal C}_k(G)\subseteq A$  and, as a consequence, $$\adim_{k}(G)\ge |{\cal C}_k(G)|.$$
\end{remark}

\begin{theorem}\label{theoAdjDimkn}
Let $G$ be a $k$-adjacency dimensional graph of order $n\ge 2$. Then $\adim_k(G)=n$ if and only if ${\cal C}_k(G)=V(G)$.
\end{theorem}

\begin{proof}
Assume that ${\cal C}_k(G)=V(G)$.  Since every $k$-adjacency dimensional graph $G$ satisfies that $\adim_k(G)\le n$, by Remark \ref{remAdjTauk} we obtain that $\adim_k(G)=n$.

Suppose that  there exists at least one vertex $x$ such that $x\not \in {\cal C}_k(G)$.  In such a case, for any  $a,b\in V(G)$ such that  $x\in {\cal C}_G(a,b)$, we have that $|{\cal C}_G(a,b)| > k$. Hence,   $|{\cal C}_G(a,b)-\{x\}| \ge k$, for all $a,b\in V(G)$ and, as a consequence,  $V(G)-\{x\}$ is a  $k$-adjacency generator for $G$, which leads to  $\adim_{k}(G)<n$. Therefore, if $\adim_{k}(G)=n$, then ${\cal C}_k(G)=V(G)$.
\end{proof}

As we will show in Propositions \ref{value-adj-Paths} and \ref{value-adj-Cycles}, $\adim_3(P_n)=n$ for $n\in\{4,\ldots,8\}$ and $\adim_4(C_n)=n$ for $n\ge 5$. These are examples of graphs satisfying conditions of Theorem \ref{theoAdjDimkn}.

\begin{corollary}\label{remarkAdjDim2n}
Let $G$ be a graph of order $n\geq 2$. Then $\adim_2(G)=n$ if and only if every vertex of $G$  belongs to a non-singleton twin equivalence class.
\end{corollary}

Since $\mathcal{C}_G(x,y)=\mathcal{C}_{\overline{G}}(x,y)$ for all $x,y\in V(G)$,  we deduce the following result, which was previously observed  for $k=1$  by  Jannesari and Omoomi in \cite{JanOmo2012}.

\begin{remark}\label{propAdj-G_Gcomp}
For any nontrivial graph $G$ and $k\in \{1,2,\ldots,\mathcal{C}(G)\}$,  $$\adim_k(G)=\adim_k(\overline{G}).$$
\end{remark}



Now we consider the limit case of the trivial bound $\adim_{k}(G)\ge k$. The case $k=1$ was studied in   \cite{JanOmo2012} where the authors showed that   $\adim_{1}(G)=1$ if and only if $G\in\{P_2,P_3,\overline{P}_2,\overline{P}_3\}$.

\begin{proposition}\label{propValueClassic2}
If $G$ is a graph of order $n\ge 2$, then  $\adim_{k}(G)=k$ if and only if $k\in \{1,2\}$ and $G\in\{P_2,P_3,\overline{P}_2,\overline{P}_3\}$
\end{proposition}

\begin{proof}
The case $k=1$ was studied in \cite{JanOmo2012}. On the other hand, by performing some simple calculations, it is straightforward to see that $\adim_{2}(G)=2$ for $G\in\{P_2,P_3,\overline{P}_2,\overline{P}_3\}$.

Now, suppose that $\adim_{k}(G)=k$ for some $k\ge 2$. By Corollary \ref{firstConsequenceAdjMonotony} we have $k=\adim_{k}(G)\ge \adim_{1}(G)+k-1$ and, as a consequence, $\adim_{1}(G)=1$. Hence, $G\in\{P_2,P_3,\overline{P}_2,\overline{P}_3\}$. Finally, since  the graphs in $\{P_2,P_3,\overline{P}_2,\overline{P}_3\}$ are $2$-adjacency dimensional, the proof is complete.
\end{proof}

According to the result above, it is interesting to study the graphs where $\adim_k(G)=k+1$. 
To begin with, we state the following remark.

\begin{remark}\label{remAd_1-3}
If $G$ is a graph of order $n\ge 7$, then $\adim_1(G)\ge 3$.
\end{remark}

\begin{proof}
Suppose for purposes of contradiction, that $\adim_1(G)\le 2$. By Proposition \ref{propValueClassic2} we deduce that $\adim_1(G)=2$. Let $B=\{u,v\}$ be an adjacency basis of $G$. Then for any  $w\in V(G)-B$ the distance  vector $(d_{G,2}(u,w),d_{G,2}(v,w))$ must belong to $\{(1,1),(1,2),(2,1),(2,2)\}$. Since $|V(G)-B|\ge 5$, by  Dirichlet's box principle at least two elements of $V(G)-B$ have the same distance vector, which is a contradiction.
Therefore,  $\adim_1(G)\ge 3$.
\end{proof}

By Corollary \ref{firstConsequenceAdjMonotony} (ii) and Remark \ref{remAd_1-3} we obtain the following result.

\begin{theorem}\label{Bound-le-k+2}
For any graph  $G$  of order  $n\ge 7$ and $k\in\{1,\ldots,\mathcal{C}(G)\}$,  $$\adim_k(G)\ge k+2.$$
\end{theorem}

From Remark \ref{remAd_1-3} and Theorem \ref{Bound-le-k+2}, we only need to consider graphs of order $n\in \{3,4,5,6\}$ to determine those satisfying $\adim_k(G)=k+1$. If $n=3$, then by Proposition \ref{propValueClassic2} we conclude that $\adim_1(G)=2$ or  $\adim_2(G)=3$ if and only if $G\in \{K_3,N_3\}$. For $k\in \{1,2\}$ and $n\in \{4,5,6\}$   the graphs satisfying  $\adim_k(G)=k+1$ 
can be determined by a simple calculation. Here we just show some of these graphs in Figure \ref{figG_iK_3}. Finally, the cases $\adim_3(G)=4$ and $\adim_5(G)=5$ are studied in the following two remarks.

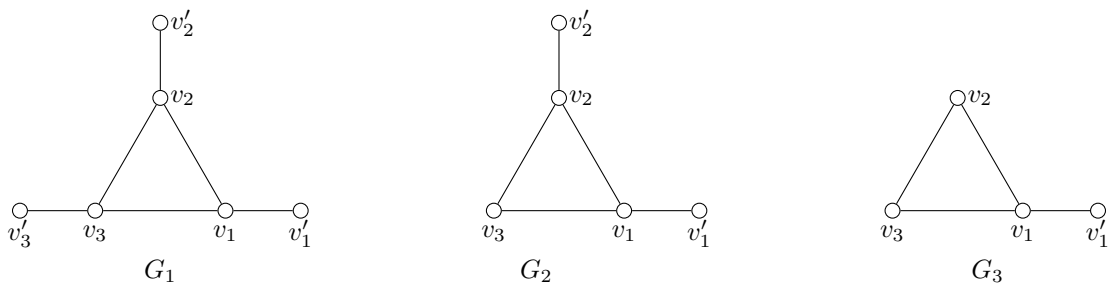
\begin{figure}[!ht]
\centering
\begin{tikzpicture}[transform shape, inner sep = .7mm]
\foreach \i in {1,2,3}
{
\foreach \ind in {1,2,3}
{
\pgfmathparse{120*\ind-150};
\node [draw=black, shape=circle, fill=white, xshift=(\i-1)*5.3 cm] (v\i\ind) at (\pgfmathresult:1 cm) {};
\ifthenelse{\ind=1 \OR \ind=3}
{
\node [scale=1] at ([yshift=-.3 cm]v\i\ind) {$v_\ind$};
}
{
\node [scale=1] at ([xshift=.3 cm]v\i\ind) {$v_\ind$};
};
\ifthenelse{\i=1 \OR \i=2 \OR \i=3 \AND \ind=1}
{
\node [draw=black, shape=circle, fill=white] (v\i4) at ([xshift=1 cm]v\i\ind) {};
\node [scale=1] at ([yshift=-.3 cm]v\i4) {$v_\ind'$};
\draw[black] (v\i\ind) -- (v\i4);
}{};
\ifthenelse{\i=1 \OR \i=2 \AND \ind=2}
{
\node [draw=black, shape=circle, fill=white] (v\i5) at ([yshift=1 cm]v\i\ind) {};
\node [scale=1] at ([xshift=.3 cm]v\i5) {$v_\ind'$};
\draw[black] (v\i\ind) -- (v\i5);
}{};
\ifthenelse{\i=1 \AND \ind=3}{
\node [draw=black, shape=circle, fill=white] (v\i6) at ([xshift=-1 cm]v\i\ind) {};
\node [scale=1] at ([yshift=-.3 cm]v\i6) {$v_\ind'$};
\draw[black] (v\i\ind) -- (v\i6);
}{};
}
\foreach \ind in {1,2}
\pgfmathparse{int(\ind+1)}
\draw[black] (v\i\ind) -- (v\i\pgfmathresult);

\draw[black] (v\i1) -- (v\i3);
\pgfmathparse{int((\i-1)*5.8)};
\node at (\pgfmathresult,-1.3) {$G_\i$};
}
\end{tikzpicture}
\caption{Any graph belonging to the families ${\cal G}_B(G_1)$, ${\cal G}_B(G_2)$ or $\{K_1\cup K_3,G_3\}$, where $B=\{v_1,v_2,v_3\}$, satisfies $\adim_2(G)=3$. The reader is referred to Section \ref{SectionFamilies} for the construction of the families ${\cal G}_B(G_i)$.}\label{figG_iK_3}
\end{figure}

The  set of nontrivial  distinctive vertices of a pair $x,y\in V(G)$, with regard to  the metric $d_{G,2}$,
will be denoted by $\mathcal{C}_G^*(x,y)=\mathcal{C}_G(x,y)-\{x,y\}.$ Notice that two vertices $x,y$ are twins if and only if
${\cal C}_G^*(x,y)=\emptyset$.

\begin{remark}\label{rem3_4}
A graph $G$ of order greater than or equal to four satisfies $\adim_3(G)=4$ if and only if $G\in \{P_4,C_5\}$.
\end{remark}

\begin{proof}
If $G\in \{P_4,C_5\}$, then it is straightforward to check that $\adim_3(G)=4$. Assume that  $B=\{v_1,\ldots,v_4\}$ is a $3$-adjacency basis of $G$. 
Since for any pair of vertices $v_i,v_j\in B$, there exists $v_l\in B\cap\mathcal{C}^*(v_i,v_j)$, by inspection we can check  that $\langle B\rangle\cong P_4$. We assume  that $v_i\sim v_{i+1}$ for $i\in\{1,2,3\}$. If $V(G)-B=\emptyset$, then $G\cong P_4$. Suppose that there exists $v\in V(G)-B$. If $v\sim v_2$, then the fact that $|B\cap\mathcal{C}^*(v,v_1)|\ge 2$ leads to $v\sim v_3$ and $v\sim v_4$. Since $|B\cap\mathcal{C}^*(v,v_4)|\ge 2$ and $v\sim v_3$, it follows that $v\sim v_1$. Thus, $v$ is connected to any vertices in $B$, which leads to $|B\cap\mathcal{C}^*(v,v_2)|=|\{v_4\}|=1$, contradicting the fact that $B$ is a $3$-adjacency basis of $G$. Analogously if $v\sim v_3$, then we arrive at the same contradiction. Thus, $v\sim v_1$ or $v\sim v_4$. If $v\sim v_1$ and $v\not\sim v_4$, then $|B\cap\mathcal{C}^*(v,v_2)|=|\{v_3\}|=1$, contradicting the fact that $B$ is a $3$-adjacency basis of $G$. Now, if $v\sim v_1$ and $v\sim v_4$, then $G\cong C_5$. If $|V(G)|\ge 6$, then there exist $u,v\in V(G)-B$. Since $|B\cap\mathcal{C}(u,v)|\ge 3$, then either $|B\cap N(u)|\ge 2$ or $|B\cap N(v)|\ge 2$. Suppose that $|B\cap N(u)|\ge 2$. As discussed earlier, $B\cap N(u)=\{v_1,v_4\}$. Since $|B\cap\mathcal{C}(u,v)|\ge 3$, it follows that either $v\sim v_2$ or $v\sim v_3$, which, as we saw earlier, contradicts the fact that $B$ is a $3$-adjacency basis of $G$.
\end{proof}

By Corollary \ref{firstConsequenceAdjMonotony} (i) and Remark \ref{rem3_4} we deduce that $\adim_4(G)\ge 6$ for any graph $G$ of order at least five such that $G\not\cong C_5$. Since $\adim_4(C_5)=5$, we obtain the following result.

\begin{remark}\label{rem4_5}
A graph $G$ of order $n\ge 5$ satisfies that $\adim_4(G)=5$ if and only if $G\cong C_5$.
\end{remark}

From Corollary \ref{firstConsequenceAdjMonotony} (i) and Remark \ref{rem4_5}, it follows that any $4$-adjacency dimensional graph $G$ of order six satisfies $\adim_4(G)=6$, as the case of $C_6$.

\subsection{Large families of graphs having a common $k$-adjacency generator} \label{SectionFamilies}

Given a $k$-adjacency basis $B$ of a graph $G=(V,E)$, we say that a graph $G'=(V,E')$ belongs to the family  ${\cal G}_B(G)$ if and only if $N_{G'}(x)=N_G(x)$, for every $x\in B$. Figure \ref{figExamplekPermutation} shows some graphs belonging to the  family  ${\cal  G}_{B}(G)$  having a common $2$-adjacency basis $B=\{v_2,v_3,v_4,v_5\}$.

\begin{figure}[!ht]
\centering
\begin{tikzpicture}[transform shape, inner sep = .7mm]
\foreach \y in {0,1}
{
\foreach \x in {7,12,17}
{
\pgfmathparse{\y*(-5.25)};
\node [draw=black, shape=circle, fill=white] (v\y\x1) at (\x,\pgfmathresult) {};
\node [scale=1] at ([yshift=-.3 cm]v\y\x1) {$v_1$};
\foreach \ind in {2,...,5}
{
\pgfmathparse{\y*(-5.25)+3.5-\ind};
\pgfmathsetmacro\posx{\x+2};
\node [draw=black, shape=circle, fill=black] (v\y\x\ind) at (\posx,\pgfmathresult) {};
\node [scale=1] at ([yshift=-.3 cm]v\y\x\ind) {$v_\ind$};
}
\foreach \ind in {2,...,5}
\draw[black] (v\y\x1) -- (v\y\x\ind);

\foreach \ind in {6,...,9}
{
\pgfmathparse{\y*(-5.25)+11.25-\ind*1.5};
\pgfmathsetmacro\posx{\x+4};
\node [draw=black, shape=circle, fill=white] (v\y\x\ind) at (\posx,\pgfmathresult) {};
\node [scale=1] at ([yshift=-.3 cm]v\y\x\ind) {$v_\ind$};
}
\foreach \ind in {2,3}
\draw[black] (v\y\x6) -- (v\y\x\ind);
\foreach \ind in {3,4}
\draw[black] (v\y\x7) -- (v\y\x\ind);
\foreach \ind in {4,5}
\draw[black] (v\y\x8) -- (v\y\x\ind);
\foreach \ind in {2,5}
\draw[black] (v\y\x9) -- (v\y\x\ind);
}
}
\node at ([yshift=-.9 cm]v075) {$G$};
\node at ([yshift=-.9 cm]v0125) {$G_1$};
\node at ([yshift=-.9 cm]v0175) {$G_2$};
\node at ([yshift=-.9 cm]v175) {$G_3$};
\node at ([yshift=-.9 cm]v1125) {$G_4$};
\node at ([yshift=-.9 cm]v1175) {$G_5$};
\draw (v0121.north) to[bend left] (v0126.north);
\draw (v0171.north) to[bend left] (v0176.north);
\draw (v0171.east) to[bend right] (v0178.west);
\draw (v176.east) to[bend left] (v177.east);
\draw (v1126.east) to[bend left] (v1127.east);
\draw (v1127.east) to[bend left] (v1128.east);
\draw (v1176.east) to[bend left] (v1177.east);
\draw (v1177.east) to[bend left] (v1178.east);
\draw (v1178.east) to[bend left] (v1179.east);
\end{tikzpicture}
\caption{\label{figExamplekPermutation}$B=\{v_2,v_3,v_4,v_5\}$ is a $2$-adjacency basis of $G$ and
$\{G,G_1,G_2,G_4,G_5\}\subset {\cal  G}_{B}(G)$.}
\end{figure}
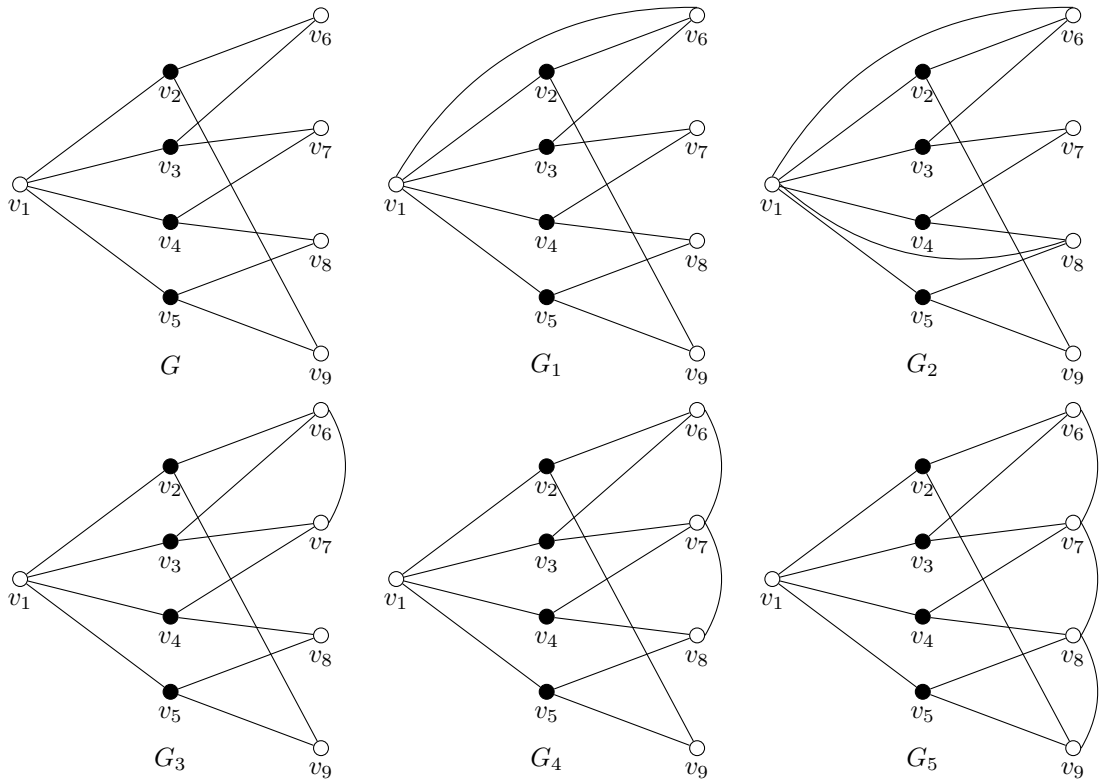

Notice that if $B \ne V(G)$, then the edge set of  any graph $G'\in {\cal{G}}_B(G) $  can be partitioned into two sets $E_1$, $E_2$, where $E_1$ consists of all edges of $G$ having at least one vertex in $B$ and $E_2$ is a subset of edges of a complete graph whose vertex set is $V(G)-B$. Hence, ${\cal{G}}_B(G)$ contains
$2^{\frac{|V(G)-B|(|V(G)-B|-1)}{2}}$ different graphs.

With the above notation in mind we can state our next result.

\begin{theorem}\label{Permutation-kadimension}
Any $k$-adjacency basis $B$ of a graph $G$ is a $k$-adjacency generator for any graph $G'\in {\cal G}_B(G)$, and as a consequence, $$\adim_k(G')\le\adim_k(G).$$
\end{theorem}

\begin{proof}
Assume that $B$ is a $k$-adjacency basis of a graph $G=(V,E)$. Let $G'=(V,E')$ such that $N_{G'}(x)=N_G(x)$, for every $x\in B$. We will show that $B$ is a $k$-adjacency generator for any graph $G'$. To this end, we take two different vertices $u,v\in V$. Since $B$ is a $k$-adjacency basis of $G$, there exists $B_{uv}\subseteq B$ such that $|B_{uv}|\ge k$ and for every $x\in B_{uv}$ we have that $d_{G,2}(x,u)\ne d_{G,2}(x,v)$. Now, since for every $x\in B_{uv}$ we have that $N_{G'}(x)=N_G(x)$, we obtain that $d_{G',2}(u,x)=d_{G,2}(u,x)\ne d_{G,2}(v,x) = d_{G',2}(v,x)$. Hence, $B$ is a $k$-adjacency generator for $G'$ and, in consequence, $|B|=\adim_k(G)\ge \adim_k(G')$.
\end{proof}

By Proposition \ref{propValueClassic2} we have that if $G$ is a graph of order $n\ge 2$, then $\adim_{k}(G)=k$ if and only if $k\in \{1,2\}$ and $G\in\{P_2,P_3,\overline{P}_2,\overline{P}_3\}$. 
Thus, for any graph $H$ of order greater than three,  $\adim_{k}(H)\ge k+1$. Therefore, the next corollary is a direct consequence of Theorem \ref{Permutation-kadimension}.

\begin{corollary}\label{Corollary-Famili-dim=k+1}
Let $B$  be a $k$-adjacency basis  of  a graph $G$ of order $n\ge 4$ and let $G'\in {\cal  G}_B(G)$. If
$\adim_k(G)=k+1$, then $\adim_k(G')=k+1.$
\end{corollary}

Our next result immediately follows from Theorems \ref{Bound-le-k+2} and \ref{Permutation-kadimension}.

\begin{theorem}
Let $B$  be a $k$-adjacency basis  of  a graph $G$ of order $n\ge 7$ and let $G'\in {\cal  G}_B(G)$. If
$\adim_k(G)=k+2$, then $\adim_k(G')=k+2.$
\end{theorem}

An example of application of the result above is shown in Figure \ref{figExamplekPermutation}, where $\adim_2(G')=4$ for all $G'\in {\cal  G}_{B}(G)$. In this case   ${\cal{G}}_B(G)$ contains
$2^{10}=1024$ different graphs.

\section{The $k$-adjacency dimension of join graphs}\label{SectionJoin}

The \emph{join} $G+H$ of two vertex-disjoint graphs $G=(V_{1},E_{1})$ and $H=(V_{2},E_{2})$ is the graph with vertex set $V(G+H)=V_{1}\cup V_{2}$ and edge set $$E(G+H)=E_{1}\cup E_{2}\cup \{uv\,:\,u\in V_{1},v\in V_{2}\}$$

Note that $D(G+H)\le 2$ and so for any pair of graphs $G$ and $H$,  $$\dim_k(G+H)=\adim_k(G+H).$$

\subsection{The particular case of $K_1 +H$}

The following remark is a particular case of Corollary \ref{remarkAdjDim2n}.

\begin{remark}\label{remarkJoinAdjK_1H}
Let $H$ be a graph of order $n$. Then $\adim_2(K_1+H)=n+1$ if and only if $\Delta(H)=n-1$ and   every  vertex $v\in V(H)$ of degree $\delta (v)<n-1$  belongs to a non-singleton twin equivalence class.
\end{remark}

For any graph  $H$,  if $x,y\in V(H)$, then $ \mathcal{C}_{K_1+H}(x,y)=\mathcal{C}_H(x,y)$. Also, if  $x\not \in V(H)$  then
$ \mathcal{C}_{K_1+H}(x,y)=\{x\}\cup (V(H)-N_H(y))$.
Hence,
$$\mathcal{C}(K_1+H)=\displaystyle\min\{\mathcal{C}(H),n-\Delta(H)+1\}.$$

\begin{proposition}\label{propLowerBoundJoinK_1-H}
Let $H$ be a graph of order $n\ge 2$ and $k\in\{1,\ldots,\mathcal{C}(K_1+H)\}$. Then $$\adim_k(K_1+H)\ge\adim_k(H).$$
\end{proposition}

\begin{proof}
Let $A$ be a $k$-adjacency basis of $K_1+H$, $A_H=A\cap V(H)$ and let $x,y\in V(H)$ be two different vertices. Since  $\mathcal{C}_{K_1+H}(x,y)=\mathcal{C}_{H}(x,y)$, it follows that $|A_H\cap\mathcal{C}_{H}(x,y)|=|A\cap\mathcal{C}_{K_1+H}(x,y)|\ge k$, and as a consequence, $A_H$ is a $k$-adjacency generator for $H$. Therefore, $\adim_k(K_1+H)=|A|\ge |A_H|\ge\adim_k(H)$.
\end{proof}

\begin{theorem}\label{propEqualJoinK_1-H}
For any     nontrivial graph $H$,  
the following assertions are equivalent: 
\begin{enumerate}[{\rm (i)}]
\item There exists a $k$-adjacency basis $A$ of $H$ such that $|A-N_H(y)|\ge k$, for all  $y\in V(H)$.
\item $\adim_k(K_1+H)=\adim_k(H).$
\end{enumerate}
\end{theorem}

\begin{proof} 
Let $A$ be a $k$-adjacency basis  of $H$ such that $|A-N_H(y)|\ge k$, for all $y\in V(H)$.  By Proposition \ref{propLowerBoundJoinK_1-H} we have that $\adim_k(K_1+H)\ge\adim_k(H)$. It remains to prove that $\adim_k(K_1+H)\le\adim_k(H)$.  We will prove that $A$ is a $k$-adjacency generator for $K_1+H$. We differentiate two cases for two vertices $x,y\in V(K_1+H)$. If $x,y\in V(H)$, then the fact that $A$ is a $k$-adjacency basis of $H$ leads to $k\le|A\cap\mathcal{C}_{H}(x,y)|=|A\cap\mathcal{C}_{K_1+H}(x,y)|$. On the other hand, if $x$ is the vertex of $K_1$ and $y\in V(H)$, then the fact that $\mathcal{C}_{K_1+H}(x,y)=\{x\}\cup(V(H)-N_H(y))$ and $|A-N_H(y)|\ge k$ leads to $|A\cap\mathcal{C}_{K_1+H}(x,y)|\ge k$. Therefore, $A$ is a $k$-adjacency generator for $K_1+H$, and as a consequence, $\adim_k(H)=|A|\ge\adim_k(K_1+H)$. 

On the other hand, let $B$ be a $k$-adjacency basis of $K_1+H$ such that $|B|=\adim_k(H)$ and let  $B_H=B\cap V(H)$. Since for any $h_1,h_2\in V(H)$ the vertex of $K_1$ does not belong to  $\mathcal{C}_{K_1+H}(h_1,h_2)$, we conclude that $B_H$ is a $k$-adjacency generator for $H$. Thus, $|B_H|= \adim_k(H)$ and, as a consequence,  $B_H$ is a $k$-adjacency basis of  $H$. If there exists  $h\in V(H)$ such that $|B_H-N_H(h)|< k$, then  
$|B\cap \mathcal{C}_{K_1+H}(v,h)|=|B_H-N_H(h)|<k$, which is a contradiction. Therefore, the result follows.
\end{proof}

  Our next result on graphs of diameter grater than or equal to six, is a direct consequence of  Theorem \ref{propEqualJoinK_1-H}.

\begin{corollary}\label{Th1HconditionNcD6}
For any  graph $H$ of diameter $D(H)\ge 6$ and $k\in\{1,\ldots,\mathcal{C}(K_1+H)\}$, $$\adim_k(K_1+H)=\adim_k(H).$$
\end{corollary}

\begin{proof}  

Let $S$ be a $k$-adjacency basis of  $H$.  We will show that $|S-N_H(x)|\ge k $, for all   $x\in V(H)$.  Suppose, for the purpose of contradiction, that there exists $x\in V(H)$ such that $|S\cap(V(H)-N_H(x))|<k$. Let $F(x)=S\cap N_H[x]$. Notice that  $|S|\ge k$ and hence $F(x)\ne\emptyset$.

From the assumptions above, if $V(H)=F(x)\cup \{x\}$, then $D(H)\le 2$, which is a contradiction. If for every $y\in V(H)-\left(F(x)\cup\{x\}\right)$ there exists $z\in F(x)$ such that $d_{H}(y,z)=1$, then $d_H(v,v')\le 4$ for all $v,v'\in V(H)-\left(F(x)\cup\{x\}\right)$. Hence $D(H)\le 4$, which is a contradiction. So, we assume that there exists a vertex $y'\in V(H)-\left(F(x)\cup\{x\}\right)$ such that $d_{H}(y',z)>1$, for every $z\in F(x)$, \textit{i.e}, $N_{H}(y')\cap F(x)=\emptyset$. If $V(H)=F(x)\cup\{x,y'\}$, then by the connectivity of $H$ we have $y'\sim x$  and, as consequence, $D(H)=2$, which is also a contradiction. Hence, $V(H)-(F(x)\cup\{x,y'\})\ne\emptyset$.
Now, for any $w\in V(H)-(F(x)\cup\{x,y'\})$ we have that $|\mathcal{C}_{H}(y',w)\cap S|\ge k$ and,
since $|S\cap\left(V(H)-N_{H}(x)\right)|<k$  and $N_{H}(y')\cap F(x)=\emptyset$, we deduce that $N_{H}(w)\cap F(x)\ne \emptyset$. From this fact and the connectivity of $H$, we obtain that $d_H(y',w)\le 5$.  Hence $D(H)\le 5$, which is also a contradiction. Therefore, if $D(H)\ge 6$, then for every $x\in V(H)$ we have that $|S\cap\left(V(H)-N_{H}(x)\right)|\ge k$.
Therefore, the result follows by Theorem \ref{propEqualJoinK_1-H}. 
\end{proof}

\begin{corollary}\label{girth>4}
Let $H$ be a  graph  of girth $\mathtt{g}(H)\ge 5$ and minimum degree $\delta(H)\ge 3$. Then for any 
$k\in\{1,\ldots,\mathcal{C}(K_1+H)\}$, $$\adim_k(K_1+H)=\adim_k(H).$$
\end{corollary}

\begin{proof}
Let $A$ be a $k$-adjacency basis of $H$ and let $x\in V(H)$ and $y\in N_H(x)$. Since $\mathtt{g}(H)\ge 5$, for any $u,v\in N_H(y)-\{x\}$ we have that ${\cal C}_H(u,v)\cap N_H[x]=\emptyset$. Also, since $|{\cal C}_H(u,v)\cap A|\ge k$, we obtain that  $|A-N_H(x)|\ge k$. Therefore,   by Theorem \ref{propEqualJoinK_1-H} we conclude the proof. 
\end{proof}

A {\em fan graph}  is defined as the join graph $K_1+P_n$, where $P_n$ is a path of order $n$, and a {\em wheel graph} is defined as the join graph $K_1+C_n$, where $C_n$ is a cycle graph of order $n$. The following closed formulae for the $k$-metric dimension of fan and wheel graphs were obtained in \cite{Estrada-Moreno2013corona,Hernando2005}. Since these graphs have diameter two, we express the result in terms of the $k$-adjacency dimension.


\begin{proposition}\label{value-dim1-fans-wheels}{\rm \cite{Hernando2005}}

\begin{enumerate}[{\rm (i)}]
\item $\adim_1(K_1+P_n)=\displaystyle \left\{\begin{array}{ll}
1, & \textrm{if $n=1$,}\\
2, & \textrm{if $n=2,3,4,5$,}\\
3, & \textrm{if $n=6$,}\\
{\left\lfloor{\frac{2n+2}{5}}\right\rfloor}, & \textrm{otherwise.}\\
\end{array}\right.
$
\item $\adim_1(K_1+C_n)=\left\{\begin{array}{ll}
3, & \textrm{if $n=3,6$,}\\
{\left\lfloor{\frac{2n+2}{5}}\right\rfloor}, & \textrm{otherwise.}\\
\end{array}\right.
$
\end{enumerate}
\end{proposition}

\begin{proposition}\label{value-dim1-paths-cycles}{\rm \cite{JanOmo2012}} For any integer $n\ge 4$,
$$\adim_1(P_n)=\adim_1(C_n)=\left\lfloor{\frac{2n+2}{5}}\right\rfloor.$$

\end{proposition}

Notice that by Propositions  \ref{value-dim1-fans-wheels} and \ref{value-dim1-paths-cycles}, for any $n\ge 4,$ $n\ne 6$, we have that
 $$\adim_1(P_n)=\adim_1(K_1+P_n)=\adim_1(C_n)=\adim_1(K_1+C_n).$$ 

In order to show the relationship between the $k$-adjacency dimension of fan (wheel) graphs and path (cycle) graphs, we state the following known results.

\begin{proposition}\label{value-fans-wheels}{\rm \cite{Estrada-Moreno2013corona}}
\begin{enumerate}[{\rm (i)}]
\item $\adim_2(K_1+P_n)=\displaystyle \left\{\begin{array}{ll}
3, & \textrm{if $n=2$,}\\
4, & \textrm{if $n=3,4,5$,}\\
{\left\lceil\frac{n+1}{2}\right\rceil }, &\textrm{if $n\ge 6.$}\\
\end{array}\right.
$ 
\item  $\adim_2(K_1+C_n)=\displaystyle \left\{\begin{array}{ll}
4, & \textrm{if $n=3,4,5,6$,}\\
{\left\lceil\frac{n}{2}\right\rceil}, & \textrm{if $ n\ge 7$.}\\
\end{array}\right.
$

\item $\adim_3(K_1+P_n)=\displaystyle \left\{\begin{array}{ll}
5, & \textrm{if $n=4,5$,}\\
{n-\left\lfloor\frac{n-4}{5}\right\rfloor}, &\textrm{if $  n\ge 6$.}\\
\end{array}\right.
$ 
\item  $\adim_3(K_1+C_n)=\displaystyle \left\{\begin{array}{ll}
5, & \textrm{if $n=5,6$,}\\
{n-\left\lfloor\frac{n}{5}\right\rfloor }, & \textrm{if $ n\ge 7.$}\\
\end{array}\right.
$

\item  $\adim_4(K_1+C_n)=\displaystyle \left\{\begin{array}{ll}
6, & \textrm{if $n=5,6$,}\\
{n }, & \textrm{if $ n\ge 7.$}\\
\end{array}\right.
$
\end{enumerate}
\end{proposition}

 By Theorem \ref{theokadjacency} we have that any path graph of order at least four is $3$-adjacency dimensional and any cycle graph of order at least five is $4$-adjacency dimensional.
From Propositions \ref{propLowerBoundJoinK_1-H}  and  \ref{value-fans-wheels}  we will derive closed formulae for the $k$-adjacency dimension  of paths (for $k\in\{2,3\}$) and cycles  (for $k\in\{2,3,4\}$).

\begin{proposition}\label{value-adj-Paths}
For any integer $n\ge 4$,
$$\adim_2(P_n)=\left\lceil\frac{n+1}{2}\right\rceil \; {and } \;\adim_3(P_n)=n-\left\lfloor\frac{n-4}{5}\right\rfloor.$$
\end{proposition}

\begin{proof} 
Let $k\in \{2,3\}$ and   $V(P_n)=\{v_1,v_2,...,v_n\}$, where    $v_i$ is adjacent to $v_{i+1}$ for every $i\in \{1,...,n-1\}$.

 We first  consider the case  $n\ge 7$.
Since ${\cal C}_{P_n}(v_1,v_2)=\{v_1,v_2,v_3\}$ and ${\cal C}_{P_n}(v_{n-1},v_{n})=\{v_{n-2},v_{n-1},v_n\}$, we deduce that for any $k$-adjacency basis $A$ of $P_n$ and any $y\in V(T)$,  $|A-N_{P_n}(y)|\ge k$. Hence,  Theorem \ref{propEqualJoinK_1-H} leads to $\adim_k(K_1+P_n)=\adim_k(P_n)$. Therefore, by Proposition \ref{value-fans-wheels} we deduce the result for $n\ge 7$.

Now, for $n=6$, since  ${\cal C}_{P_6}(v_1,v_2)=\{v_1,v_2,v_3\}$ and ${\cal C}_{P_6}(v_{5},v_{6})=\{v_{4},v_{5},v_6\}$, we deduce that $\adim_2(P_6)\ge 4$ and $\adim_3(P_6)=6$. In addition, $\{v_1,v_3,v_4,v_6\}$ is a $2$-adjacency generator for $P_6$ and so  $\adim_2(P_6)=4$. 

From now on, let  $n\in\{4,5\}$.   By Proposition \ref{propLowerBoundJoinK_1-H} we have $\dim_k(K_1+P_n) \ge\adim_k(P_n)$. It remains to prove that $\adim_k(K_1+P_n)\le\adim_k(P_n)$.

If $n=4$ or $n=5$, then by Proposition \ref{propValueClassic2}, $\adim_2(P_n)\ge 3$. Note that $\{v_1,v_2,v_4\}$ and $\{v_1,v_3,v_5\}$ are $2$-adjacency generators for $P_4$ and $P_5$, respectively. Thus, $\adim_2(P_4)=\adim_2(P_5)=3$. Let $A$ be a $3$-adjacency basis of $P_n$, where $n\in\{4,5\}$. Since $\mathcal{C}_{P_n}(v_1,v_2)=\{v_1,v_2,v_3\}$ and $\mathcal{C}_{P_n}(v_{n-1},v_n)=\{v_{n-2},v_{n-1},v_n\}$, we have that $(A\cap\mathcal{C}_{P_n}(v_1,v_2))\cup(A\cap\mathcal{C}_{P_n}(v_{n-1},v_n))=V(P_n)$, and as consequence, $A=V(P_n)$. Therefore, $\adim_3(P_4)=4$ and $\adim_3(P_5)=5$ and, as a consequence, the result follows.
\end{proof}


\begin{proposition}\label{value-adj-Cycles}
For any integer $n\ge 5$,
$$\adim_2(C_n)=\left\lceil\frac{n}{2}\right\rceil,\;  \adim_3(C_n)= n-\left\lfloor\frac{n}{5}\right\rfloor\; {and} \;
\adim_4(C_n)=n.$$
\end{proposition}

\begin{proof} 
Let $k\in \{2,3,4\}$ and  $V(C_n)=\{v_1,v_2,...,v_n\}$, where    $v_i$ is adjacent to $v_{i+1}$ and the subscripts are taken modulo $n$.
 
 We first  consider the case  $n\ge 7$.
Since ${\cal C}_{C_n}(v_{i+3},v_{i+4})=\{v_{i+2},v_{i+3},v_{i+4},v_{i+5}\}$, we deduce that for any $k$-adjacency basis $A$ of $C_n$,  $|A-N_{C_n}(v_i)|\ge k$. Hence,  Theorem \ref{propEqualJoinK_1-H} leads to $\adim_k(K_1+C_n)=\adim_k(C_n)$. Therefore, by Proposition \ref{value-fans-wheels}    we deduce the result for $n\ge 7$.

From now on, let $n\in\{5,6\}$.   By Proposition \ref{propLowerBoundJoinK_1-H} we have $\dim_k(K_1+G) \ge\adim_k(G)$. It remains to prove that $\adim_k(K_1+H)\le\adim_k(H)$.

 By Theorem \ref{theoMonotyForAdjKs}, we deduce that $2=\adim_1(C_5)<\adim_2(C_5)<\adim_3(C_5)<\adim_4(C_5)\le 5$. Hence, $\adim_2(C_5)=3$, $\adim_3(C_5)=4$ and $\adim_4(C_5)= 5$. Therefore, for $n=5$  the result follows.

By Theorem \ref{theoMonotyForAdjKs}, $\adim_2(C_6)>\adim_1(C_6)=2$ and, since $\{v_1,v_3,v_5\}$ is a $2$-adjacency generator for $C_6$, we obtain that  $\adim_2(C_6)=3$. Now, let $A_4$ be a $4$-adjacency basis of $C_6$. If $|A_4|\le 5$, then there exists at least one vertex which does not belong to $A_4$, say $v_1$. Then, $|\mathcal{C}_{C_n}(v_1,v_2)\cap A_4|\le 3$, which is a contradiction. Thus, $\adim_4(C_6)=|A_4|=6$. Let $A_3^1=\{v_1,v_2,v_3,v_4\}$, $A_3^2=\{v_1,v_2,v_3,v_5\}$ and $A_3^3=\{v_1,v_2,v_4,v_5\}$. Note that any  manner of selecting four different vertices from $C_6$ is equivalent to some of these $A_3^1, A_3^2, A_3^3$. Since $|\mathcal{C}_{C_n}(v_5,v_6)\cap A_3^1|=|\{v_1,v_4\}|=2<3$, $|\mathcal{C}_{C_n}(v_4,v_6)\cap A_3^2|=|\{v_1,v_3\}|=2<3$ and $|\mathcal{C}_{C_n}(v_1,v_2)\cap A_3^3|=|\{v_1,v_2\}|=2<3$,  we deduce that $\adim_3(C_6)\ge 5>|A_3^1|=|A_3^2|=|A_3^3|=4$. By Theorem \ref{theoMonotyForAdjKs}, $5\le \adim_3(C_6)<\adim_4(C_6)\le 6$. Thus,  $\adim_3(C_6)=5$  and, as a consequence, the result follows.
\end{proof}


By  Propositions \ref{value-dim1-fans-wheels}, \ref{value-dim1-paths-cycles}, \ref{value-fans-wheels}, \ref{value-adj-Paths} and \ref{value-adj-Cycles}  we observe that for any  $k\in\{1, 2,3\}$ and $n\ge7$, $\adim_k(K_1+P_n)=\adim_k(P_n)$  and for any  $k\in\{1,2,3,4\}$,  $\adim_k(K_1+C_n)=\adim_k(C_n)$. The next result is devoted to characterize the trees where $\adim_k(K_1+T)=\adim_k(T)$.


\begin{proposition}
Let $T$ be a tree. The following statements hold.
\begin{enumerate}[{\rm(a)}]
\item $\adim_1(K_1+T)=\adim_1(T)$ if and only if   $T\not\in {\cal F}_1= \{P_2, P_3, P_6, K_{1,n},T'\}$, where  $n\ge 3$ and $T'$ is obtained from $P_5\cup \{K_1\}$ by  joining by an edge  the vertex of $K_1$   to the central vertex of $P_5$.
\item $\adim_2(K_1+T)=\adim_2(T)$ if and only if $T\not\in {\cal F}_2=\{P_r,K_{1,n},T'\}$, where $r\in\{2,\ldots,5\}$,  $n\ge 3$ and $T'$ is a  graph obtained from $K_{1,n}\cup K_2$  by    joining by an edge one leaf of $K_{1,n}$ to one leaf of $K_2$.
\item $\adim_3(K_1+T)=\adim_3(T)$ if and only if $T\not\in {\cal F}_3=\{P_4,P_5\}$.
\end{enumerate}
\end{proposition}

\begin{proof}
For any $k\in \{1,2,3\}$ and $T\in {\cal F}_k$,   a simple inspection shows that $\adim_k(K_1+T)\ne \adim_k(T)$.  From now on, assume that   $T\not \in {\cal F}_k$, for $k\in \{1,2,3\}$, and let $\Ext(T)$ be the number of exterior major vertices of $T$. We differentiate the following three cases.\\

\noindent
Case 1. $T=P_n$. The result is a direct consequence of combining   Propositions \ref{value-dim1-fans-wheels} and \ref{value-dim1-paths-cycles} for $k=1$ and Propositions   \ref{value-fans-wheels} and \ref{value-adj-Paths} for $k>1$.\\
\\
In the following cases we shall  show that   there exists  a  $k$-adjacency basis 
 $A$ of $T$ such that $|A-N_T(v)|\ge k$, for all $v\in V(T)$. Therefore, the result follows by Theorem \ref{propEqualJoinK_1-H}.
 \\

\noindent 
Case 2. $\Ext(T)=1$. Let $u$ be the only exterior major vertex of $T$. 

We first take $k=1$. Since any two vertices adjacent to $u$ must be distinguished by at least one vertex, we have that all paths from $u$ to its terminal vertices, except at most one, contain at least one vertex in $A$. Thus,  $|A-N_T(y)|\ge 1$, for all $y\in V(T)-\{u\}$. Now we shall show that $|A-N_T(u)|\ge 1$. If $u\in A$ or $A\not \subseteq N_T(u)$, then we are done, so we suppose that for any adjacency basis $A$ of $T$,  $u\not \in A$ and $A \subseteq N_T(u)$. If there exists a leaf $v$ such that $d_T(u,v)\ge 4$, then  
the support $v'$ of $v$ satisfies ${\cal C}_T(v,v')\cap A=\emptyset$, which is a contradiction. Hence, the eccentricity of $u$ satisfies $2\le \epsilon(u)\le 3$. 
If   $w$ is a leaf of $T$ such that $d_T(u,w)=\epsilon(u)$, then the vertex $u'\in N_T(u)$ belonging to the path from $u$ to $w$ must belong to $A$ and, as a consequence  $A'=(A-\{u'\})\cup \{w\}$ is an adjacency basis of $T$, which is a contradiction. 

We now take  $k=2$. Let $A$ be a $2$-adjacency basis of $T$. Since any two vertices adjacent to $u$ must be distinguished by at least two vertices in $A$,  either all paths joining $u$ to its terminal vertices contain at least one vertex of $A$ or all but one contain at least two vertices of $A$. Thus, any vertex $y\in V(T)-\{u\}$ and any $2$-adjacency basis $A$ of $T$ satisfy that $|A-N_T(y)|\ge 2$.

If there exist two vertices  $v,v'\in V(T)$  such that $d_T(u,v)\ge 3$ and $d_T(u,v')\ge 3$, then $|A-N_T(u)|\ge 2$, as $|A\cap {\cal C}_2(v,v')|\ge 2$. On the other hand, if  there exists only one leaf $v$  such that $d_T(u,v)\ge 3$ and another leaf $w$ such that $d_T(u,w)= 2$, we have that in order to distinguish $v$ and it support as well as $w$ and its support, $|A\cap N_T[v]| \ge 1$ and $|A\cap \{u,w\}| \ge 1$ and, as a result, $|A-N_T(u)|\ge 2$. Now, since $T\not \in {\cal F}_2$ it remains to consider the case where $u$ has eccentricity two. Let $v,w$ be two leaves such that    $d_T(u,v)=d_T(u,w)=2$. If $|N_T(u)|=3$, then the set $A$ composed by $u$ and its three terminal vertices is a $2$-adjacency basis of $T$ such that $|A-N_T(u)|\ge 2$. Assume that $|N_T(u)|\ge 4$. 
In order to distinguish $v$ and its support vertex $v'$, as well as $w$ and its support vertex $w'$, any $2$-adjacency basis $A$ of $T$ must contain at least two vertices of $\{u,v,v'\}$ and at least two vertices of $\{u,w,w'\}$. If $u\notin A$, then $v,w\in A$, and as a consequence, $|A-N_T(u)|\ge 2$. 
Assume that $u\in A$. In this case, if $A-N_T[u]\ne \emptyset$, then $|A-N_T(u)|\ge 2$.
Otherwise,  $ A\subseteq N_T[u]$ and $\{u,v',w'\}\subset A$ and, as a consequence, $A'=(A-\{v'\})\cup \{v\}$ is a $2$-adjacency basis of $T$ and $|A'-N_T(u)|\ge 2$.

\noindent
Finally, suppose that there exists exactly one leaf $v$ such that $d_T(u,v)=2$. Let $v'$ be the support vertex of $v$. In this case, $V(T)-\{v'\}$ is a $2$-adjacency basis $A$ of $T$ such that $|A-N_T(u)|\ge 2$. 

\noindent
We now take $k=3$. In this case, there exist two leaves  $v,w$   such that $d_T(u,v)\ge 2$ and $d_T(u,w)\ge 2$. Since $v$ and its support vertex $v'$ must be distinguished by at least three vertices, they must belong to any $3$-adjacency basis. Analogously, $w$ and its support vertex $w'$ must belong to any $3$-adjacency basis. In general, any leaf that is not adjacent to $u$ and its support vertex belong to any $3$-adjacency basis of $T$. Moreover, there exists at most one terminal vertex $x$ adjacent to $u$. If $x$ exists, it must be distinguished from any vertex belonging to $N_T(u)-\{x\}$ by at least three vertices. Thus, they must belong to any $3$-adjacency basis. Any vertex $y$ different from $u$ and any $3$-adjacency basis $A$ of $T$ satisfy $v,v'\in A-N_T(y)$ or $w,w'\in A-N_T(y)$. If $v,v'\in A-N_T(y)$ and $w,w'\in A-N_T(y)$, then $|A-N_T(y)|\ge 3$. Otherwise, assuming without loss of generality that $v,v'\in A-N_T(y)$, there exists a terminal vertex $z$ different from $w$ such that $y\not\sim z$. Thus, again $|A-N_T(y)|\ge 3$. If $d_T(u,v)=2$, then $v,v'$ are distinguished only by $u,v,v'$, so $u$ must belong to any $3$-adjacency basis of $T$. Thus, for any $3$-adjacency basis $A$ of $T$ we have that $u,v,w\in A-N_T(u)$, and as a consequence, $|A-N_T(u)|\ge 3$. Finally, if $d_T(u,v)>2$ and $d_T(u,w)>2$, then $v,v',w,w'\in A-N_T(u)$. Hence $|A-N_T(u)|\ge 3$. 
\\

\noindent Case 3. $\Ext(T)\ge 2$. In this case, there are at least two exterior major vertices $u,v$ of $T$ having terminal degree at least two. Let $u_1,u_2$ be two terminal vertices of $u$ and $v_1,v_2$ be two terminal vertices of $v$. Let $u_1'$ and $u_2'$ be the vertices adjacent to $u$ in the paths $u-u_1$ and $u-u_2$, respectively. Likewise, let $v_1'$ and $v_2'$ be the vertices adjacent to $v$ in the paths $v-v_1$ and $v-v_2$, respectively. Notice that it is possible that $u_1=u_1'$, $u_2=u_2'$, $v_1=v_1'$ or $v_2=v_2'$. Note also that $\mathcal{C}(u_1',u_2')=(N_T[u_1']\cup N_T[u_2'])-\{u\}$ and $\mathcal{C}(v_1',v_2')=(N_T[v_1']\cup N_T[v_2'])-\{v\}$. Since for any $k$-adjacency basis $A$ of $T$ it holds that  $|\mathcal{C}(u_1',u_2')\cap A|\ge k$ and $|\mathcal{C}(v_1',v_2')\cap A|\ge k$, and for any vertex $w\in V(T)$ we have that $(A-N_T(w))\cap\mathcal{C}(u_1',u_2')=\emptyset$ or $(A-N_T(w))\cap\mathcal{C}(v_1',v_2')=\emptyset$, we conclude that $|A-N_T(w)|\ge k$. 
\end{proof}


From now on, we are going to study some cases where $\adim_k(K_1+H)>\adim_k(H)$. 
First of all, notice that by Corollary \ref{Th1HconditionNcD6}, if $H$ is a connected graph and $\adim_k(K_1+H)\ge\adim_k(H)+1$, then     $D(H)\le 5$ and, 
 by Corollary \ref{girth>4}, if $H$ has minimum degree $\delta(H)\ge 3$, then it has girth $\mathtt{g}(H)\le 4$. 
We would point out the following consequence of Theorem \ref{propEqualJoinK_1-H}. 

\begin{corollary}
If $\adim_k(K_1+H)\ge \adim_k(H)+1$, then either $H$ is connected or $H$ has exactly two connected components,  one of which is an isolated vertex. 
\end{corollary}

\begin{proof}

Let $A$ be a $k$-adjacency basis of $H$.  
We differentiate three cases for $H$.

Case 1.  There are two connected components $H_1$ and $H_2$   of $H$  such that $|V(H_1)|\ge 2$ and $|V(H_2)|\ge 2$. As for any $i\in \{1,2\}$ and  $u,v\in V(H_i)$, $|{C}_{H}(u,v)\cap A|=|{C}_{H_i}(u,v)\cap A|\ge k$ we deduce that  $|A\cap  V(H_1)|\ge k$ and $|A\cap  V(H_2)|\ge k$. Hence, if $x\in V(H_1)$, then $|A- N_H(x)|\ge |A\cap V(H_2)|\ge k$ and if $x\in V(H)-V(H_1)$, then $|A- N_H(x)|\ge |A\cap V(H_1)|\ge k$. Thus, by Theorem \ref{propEqualJoinK_1-H}, $\adim_k(K_1+H)=\adim_k(H)$. 

Case 2. There is a connected component $H_1$ of $H$ such that $|V(H_1)|\ge 2$ and there are two isolated vertices $u, v\in V(H)$. 
From ${C}_{H}(u,v)=\{u,v\}$ we conclude that $k\le 2$ and $|\{u,v\}\cap  A|\ge k$. Moreover, for any $x,y\in V(H_1)$, $x\ne y$, we have that $|{C}_{H}(x,y)\cap A|=|{C}_{H_1}(u,v)\cap A|\ge k$ and so  $|A\cap  V(H_1)|\ge k$. Hence, if $x\in V(H_1)$, then $|A- N_H(x)|\ge |\{u,v\}\cap A|\ge k$ and if $x\in V(H)-V(H_1)$, then  $|A- N_H(x)|\ge |A\cap V(H_1)|\ge k$. Thus, by Theorem \ref{propEqualJoinK_1-H}, $\adim_k(K_1+H)=\adim_k(H)$.

Case 3. $H\cong N_n$, for $n\ge 2$. In this case $k\in \{1,2\}$,     $\adim_1(K_1+N_n)= \adim_1(N_n)=n-1$ and $\adim_2(K_1+N_n)= \adim_2(N_n)=n$.

Therefore, according to the three cases above, the result follows. 
\end{proof}

By  Proposition \ref{propLowerBoundJoinK_1-H} and Theorem  \ref{propEqualJoinK_1-H}, $\adim_k(K_1+H)\ge\adim_k(H)+1$ if and only if for any $k$-adjacency basis $A$ of $H$, there exists $h\in V(H)$ such that $|A- N_H(h)|<k$.  
 Consider,  for instance, the graph $G$ showed in Figure \ref{figExamplekPermutation}. The only $2$-adjacency basis of $G$ is $B=\{v_2,v_3,v_4,v_5\}$ and $|B-N_G(v_1)|=0$, so $\adim_2(K_1+G)\ge \adim_2(G)+1=5$. It is easy to check that $A=\{v_1,v_6,v_7,v_8,v_9\}$ is a $2$-adjacency generator for $K_1+G$,  and so $\adim_2(K_1+G)=\adim_2(G)+1=5.$ We emphasize that neither $B\cup \{v_1\}$ nor $B\cup \{x\}$  are $2$-adjacency bases of $\langle x\rangle +G$. 

\begin{proposition}\label{propEqualityK_1-H-dim_1}
Let $H$ be a graph of order $n\ge 2$ and let $k\in \{1,...,\mathcal{C}(K_1+H)\}$. If for any $k$-adjacency basis $A$ of $H$, there exists $h\in V(H)$ such that $|A- N_H(h)|=k-1$ and   $|A- N_H(h')|\ge k-1$, for all $h'\in V(H)$, then  $$\adim_k(K_1+H)=\adim_k(H)+1.$$
\end{proposition}

\begin{proof}
If   for any $k$-adjacency basis $A$ of $H$, there exists $h\in V(H)$ such that $|A- N_H(h)|=k-1$, then by Theorem \ref{propEqualJoinK_1-H}, $\adim_k(K_1+H)\ge \adim_k(H)+1.$ 

Now, let $A$ be a $k$-adjacency basis  of $H$ and let $v$ be the vertex of $K_1$. Since $|A- N_H(h')|\ge k-1$, for all $h'\in V(H)$, the set $A\cup \{v\}$, is a $k$-adjacency generator for  $K_1+H$ and, as a consequence, $\adim_k(K_1+H)\le |A\cup \{v\}|= \adim_k(H)+1.$ 
\end{proof}


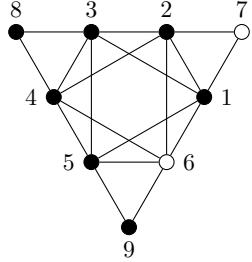
\begin{figure}[!ht]
\centering
\begin{tikzpicture}[transform shape, inner sep = .7mm]
\def\radius{1} 
\foreach \ind in {1,...,6}
{
\pgfmathparse{60*(\ind-1)};
\ifthenelse{\ind<6}
{
\node [draw=black, shape=circle, fill=black] (v\ind) at (\pgfmathresult:\radius cm) {};
}
{
\node [draw=black, shape=circle, fill=white] (v\ind) at (\pgfmathresult:\radius cm) {};
};
\ifthenelse{\ind=1\OR \ind=6}
{
\node [scale=.9] at ([xshift=.3 cm]v\ind) {$\ind$};
}
{
\ifthenelse{\ind=2\OR \ind=3}
{
\node [scale=.9] at ([yshift=.3 cm]v\ind) {$\ind$};
}
{
\node [scale=.9] at ([xshift=-.3 cm]v\ind) {$\ind$};
};
};
}
\foreach \ind in {7,...,9}
{
\pgfmathparse{30+120*(\ind-7)};
\pgfmathsetmacro\radiusE{sqrt(3)*\radius};
\ifthenelse{\ind>7}
{
\node [draw=black, shape=circle, fill=black] (v\ind) at (\pgfmathresult: \radiusE cm) {};
}
{
\node [draw=black, shape=circle, fill=white] (v\ind) at (\pgfmathresult: \radiusE cm) {};
};
\ifthenelse{\ind=7\OR \ind=8}
{
\node [scale=.9] at ([yshift=.3 cm]v\ind) {$\ind$};
}
{
\node [scale=.9] at ([yshift=-.3 cm]v\ind) {$\ind$};
};
}
\foreach \ind in {2,3,5,6,7}
\draw[black] (v1) -- (v\ind);
\foreach \ind in {3,4,6,7}
\draw[black] (v2) -- (v\ind);
\foreach \ind in {4,5,8}
\draw[black] (v3) -- (v\ind);
\foreach \ind in {5,6,8}
\draw[black] (v4) -- (v\ind);
\foreach \ind in {6,9}
\draw[black] (v5) -- (v\ind);
\draw[black] (v6) -- (v9);
\end{tikzpicture}
\caption{The set $B=\{1, 2, 3, 4, 5, 8, 9\}$ is a $3$-adjacency basis of this graph.}
\label{figFail}
\end{figure}

 The graph $H$ shown in Figure \ref{figFail} has six $3$-adjacency basis. For instance, one of them is  $B=\{1, 2, 3, 4, 5, 8, 9\}$ and the remaining ones can be found by symmetry. Notice that for any $3$-adjacency basis, say  $A$, there are two vertices $i,j$ such that $|A-N_H(i)|=2$, $|A-N_H(j)|=2$ and $|A-N_H(l)|\ge 3$, for all $l\ne i,j$. In particular, for the basis  $B$  we have  $i=3$ and $j=4$. Therefore, Proposition  \ref{propEqualityK_1-H-dim_1} leads to  $\adim_3(K_1+H)=\adim_3(H)+1=8.$

By Theorem \ref{propEqualJoinK_1-H}   and Proposition  \ref{propEqualityK_1-H-dim_1} 
we deduce the following result previously obtained  in \cite{JanOmo2012}. 

\begin{proposition}{\rm \cite{JanOmo2012}}\label{ResultOmoomiDimK1+H}
Let $H$ be graph of order $n\ge 2$. If for any adjacency basis $A$ of $H$, there exists $h\in V(H)-A$ such that $A\subseteq  N_H(h)$,  then  $$\adim_1(K_1+H)=\adim_1(H)+1,$$ otherwise,
$$\adim_1(K_1+H)=\adim_1(H).$$ 
\end{proposition}

\begin{theorem}\label{UpperBound2}
For any  nontrivial graph  $H$,    $$\adim_2(K_1+H)\le\adim_2(H)+2.$$
\end{theorem}

\begin{proof}
Let $A$ be a $2$-adjacency basis of $H$ and let $u$ be the vertex of $K_1$. Notice that  there exists at most one vertex $x\in V(H)$ such that $A\subseteq N_H(x)$. Now, if $|A- N_H(v)|\ge 1$ for all $v\in V(H)$, then we define $X=A\cup \{u\}$ and, if  there exists  $x\in V(H)$ such that $A\subseteq N_H(x)$, then we define $X=A\cup \{x,u\}$. We claim that $X$ is a $2$-adjacency generator for $K_1+H$. To show this, we first note that for any  $y\in V(H)$ we have that $|\mathcal{C}_{K_1+H}(u,y)\cap X|=|((A-N_H(y))\cup\{u\})\cap X|\ge 2$. Moreover, for any $a,b\in V(H)$ we have that $\mathcal{C}_{K_1+H}(a,b)=\mathcal{C}_{H}(a,b)$. Therefore, $X$ is a $2$-adjacency generator for $K_1+H$ and, as a consequence, $\adim_2(K_1+H) \le \adim_2(H)+2$. 
\end{proof}

We would point out that if for any $2$-adjacency basis $A$ of a graph $H$, there exists a vertex $x$ such that $A\subseteq N_H(x)$, then not necessarily $\adim_2(K_1+H)=\adim_2(H)+2$. To see this, consider the graph $G$ shown Figure \ref{figExamplekPermutation}, where $\{v_2,v_3,v_4,v_5\}$ is the only $2$-adjacency basis of $G$ and $\{v_2,v_3,v_4,v_5\}\subseteq N_H(v_1)$. However, $\{v_1,v_6,v_7,v_8,v_9\}$ is a $2$-adjacency basis of $K_1+G$ and so $\adim_2(K_1+H)=\adim(H)+1$.  Now, we prove some results showing that the inequality given in Theorem \ref{UpperBound2} is tight. 

\begin{theorem}\label{Equal2-2}
Let $H$ be a nontrivial graph. If there exists a vertex $x$ of degree $\delta(x)=|V(H)|-1$ not belonging to any $2$-adjacency basis of $H$,  then $$\adim_2(K_1+H)=\adim_2(H)+2.$$
\end{theorem}

\begin{proof}
Let $u$ be the vertex of $K_1$ and let $x\in V(H)$ be a vertex  of degree $\delta(x)=|V(H)|-1$ not belonging to any $2$-adjacency basis of $H$. In such a case, $\mathcal{C}_{K_1+H}(x,u)=\{x,u\}$ and, as a result,  both  $x$ and $u$ must belong to any $2$-adjacency basis $X$ of $K_1+H$.  Since $X-\{u\}$ is a $2$-adjacency generator for $H$ and $x\in X-\{u\}$ we conclude that $|X-\{u\}|\ge \adim_2(H)+1$  and so  $\adim_2(K_1+H)=|X|\ge \adim_2(H)+2$. By Theorem \ref{UpperBound2} we conclude the proof.
\end{proof}

Examples of graphs satisfying the premises of Theorem \ref{Equal2-2} are the fan graphs $F_{1,n}=K_1+P_n$ and the wheel graphs $W_{1,n}=K_1+C_n$ for $n\ge 7$. For these graphs we have $\adim_2(K_1+F_{1,n})=\adim_2(F_{1,n})+2$ and $\adim_2(K_1+W_{1,n})=\adim_2(W_{1,n})+2.$

\begin{theorem}
Let $H$ be a graph having an isolated vertex $v$ and a vertex $u$ of degree $\delta(x)=|V(H)|-2$. If for any $2$-adjacency basis $B$ of $H$,  neither $u$ nor $v$ belongs to $B$, then  $$\adim_2(K_1+H)=\adim_2(H)+2.$$
\end{theorem}

\begin{proof}
Let $u$ be the vertex of $K_1$. Since    $\mathcal{C}_{K_1+H}(x,u)=\{x,u,v\}$,  at least two vertices of $\{x,u,v\}$ must belong to any $2$-adjacency basis $X$ of $K_1+H$. Then we have that   $x\in X-\{u\}$ or $v\in X-\{u\}$. Since $X-\{u\}$ is a $2$-adjacency generator for $H$,  we conclude that if $|X\cap \{x,v\}|=1$, then $\adim_2(K_1+H)>|X-\{u\}|\ge \adim_2(H)+1$,  whereas if  $|X\cap \{x,v\}|=2$, then   $\adim_2(K_1+H)\ge  |X-\{u\}|\ge \adim_2(H)+2$. Hence,   $\adim_2(K_1+H)=|X|\ge\adim_2(H)+2$. By Theorem \ref{UpperBound2} we conclude the proof. 
\end{proof}

For instance, we take a family of graphs ${\cal G}=\{G_1,G_2,...\}$ such that for any  $G_i\in {\cal G}$, every vertex in $V(G_i)$ belongs to a non-singleton true twin equivalence class.  Then $X=\bigcup_{G_i\in {\cal G}}V(G_i)$ is the only $2$-adjacency basis of $H=K_1\cup (K_1+\bigcup_{G_i\in {\cal G}}G_i)$. Therefore, $\adim_2(K_1+H)=\adim_2(H)+2.$

\begin{proposition}\label{Proposition-holds_conjecture}
Let $H$ be graph and $k\in\{1,\ldots,\mathcal{C}(K_1+H)\}$. If there exists a vertex $x\in V(H)$ and a $k$-adjacency basis $A$ of $H$ such that $A\subseteq N_H(x)$, then   $$\adim_k(K_1+H)\le\adim_k(H)+k.$$
\end{proposition}

\begin{proof}
Let $u$ be the vertex of $K_1$ and assume that there exists   a vertex $v_1\in V(H)$ and a $k$-adjacency basis $A$ of $H$ such that $A\subseteq N_H(v_1)$. Since $k\le |V(H)|-\Delta(H)+1$, we have that $|V(H)-  N_H(v_1)|\ge k-1$.
 With this fact in mind, we shall show that $X=A\cup \{u\}\cup A'$  is a $k$-adjacency generator  for $K_1+H$, where  $A'=\emptyset$ if  $k=1$ and    $A'=\{v_1,v_2,...,v_{k-1}\}\subset V(H)-  N_H(v_1)$ if $k\ge 2$.  
To  this end  we only need to check that $|{\cal C}_{K_1+H}(u,v)\cap X|\ge k$, for all $v\in V(H)$. 
On one hand, $|{\cal C}_{K_1+H}(u,v_1)\cap X|=|\{u\}\cup A'|= k$.
On the other hand, since $A\subseteq N_H(v_1)$, for any $v\in V(H)-\{v_1\}$ we have that $|A- N_H(v)|\ge k$ and, as a consequence, $|{\cal C}_{K_1+H}(u,v)\cap X|\ge k$. Therefore, 
$X$ is a $k$-adjacency generator for $K_1+H$ and, as a result, $\adim_k(K_1+H)\le|X|=\adim_k(H)+k.$
\end{proof}

\vspace{-0.5cm}
\begin{figure}[!ht]
\centering
\begin{tikzpicture}[transform shape, inner sep = .7mm]
\def\radius{1.5}
\foreach \ind in {1,...,9}
{
\pgfmathparse{40*(\ind-1)};
\ifthenelse{\ind=2\OR \ind=3\OR \ind=5\OR \ind=6\OR \ind=7\OR \ind=9}
{
\node [draw=black, shape=circle, fill=black] (v\ind) at (\pgfmathresult:\radius cm) {};
}
{
\node [draw=black, shape=circle, fill=white] (v\ind) at (\pgfmathresult:\radius cm) {};
};
\ifthenelse{\ind>1}
{
\pgfmathparse{int(\ind-1)};
\draw[black] (v\pgfmathresult) -- (v\ind);
}
{};
\ifthenelse{\ind=1\OR \ind=9}
{
\node [scale=.9] at ([xshift=.3 cm]v\ind) {$\ind$};
}
{
\ifthenelse{\ind=2\OR \ind=3\OR \ind=4}
{
\node [scale=.9] at ([yshift=.3 cm]v\ind) {$\ind$};
}
{
\ifthenelse{\ind=7\OR \ind=8}
{
\node [scale=.9] at ([yshift=-.3 cm]v\ind) {$\ind$};
}
{
\node [scale=.9] at ([xshift=-.3 cm]v\ind) {$\ind$};
};
};
};
}
\foreach \ind in {3,5,6,7,8,9}
{
\draw[black] (v1) -- (v\ind);
}
\foreach \ind in {6,7,8}
{
\draw[black] (v2) -- (v\ind);
}
\end{tikzpicture}
\caption{The set $A=\{2, 3, 5, 6, 7, 9\}$ is the only $3$-adjacency basis of $H$ and $A\subset N_H(1)$.
}
\label{figEjemplo3plus3}
\end{figure}
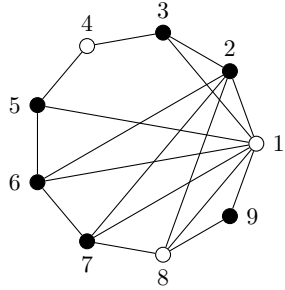

The bound above is tight. It is achieved, for instance, for the graph shown in Figure \ref{figEjemplo3plus3}. 
In this case $\adim_3(K_1 +H)=\adim_3(H)+3=9$. The set $\{2, 3, 5, 6, 7, 9\}$ is the only $3$-adjacency basis of $H$, whereas $\langle u\rangle+H$ has four $3$-adjacency bases, \textit{i.e.,} $\{1, 2, 3, 4, 5, 6, 7, 8, u\}$, $\{1, 2, 3, 4, 5, 6, 7, 9, u\}$  $ \{1, 2, 3, 4, 5, 7, 8, 9, u\}$ and $\{1, 2, 3, 4, 6, 7, 8, 9, u\}$.

\begin{conjecture} \label{ref_conjecture}
Let $H$ be graph of order $n\ge 2$ and $k\in\{1,\ldots,\mathcal{C}(K_1+H)\}$. Then $$\adim_k(K_1+H)\le\adim_k(H)+k.$$
\end{conjecture}

We have shown that Conjecture \ref{ref_conjecture} is true for any graph $H$ and $k\in \{1,2\}$, and for any $H$ and $k$ satisfying the premises of Proposition \ref{Proposition-holds_conjecture}. Moreover, in order to assess the potential validity of Conjecture~\ref{ref_conjecture}, we explored the entire set of  graphs of order $n \le 11$ and minimum degree two by means of an exhaustive search algorithm. This search yielded no graph $H$ such that $\adim_k(K_1 + H)>\adim_k(H) + k$, $k \in \{3,4\}$, a fact that empirically supports our conjecture.

\subsection{The $k$-adjacency dimension of  $G +H$ for $G\not\cong K_1$ and $H\not\cong K_1$}

Two different vertices $u,v$ of $G+H$ belong  to the same twin equivalence class if and only if at least one of the following three statements hold.
\begin{enumerate}[(a)]
\item $u,v\in V(G)$ and $u,v$ belong to the same twin equivalence class of $G$.
\item $u,v\in V(H)$ and $u,v$ belong to the same twin equivalence class of $H$.
\item $u\in V(G)$, $v\in V(H)$, $N_G[u]=V(G)$ and $N_H[v]=V(H)$.
\end{enumerate}
The following two remarks are  direct consequence of Corollary \ref{remarkAdjDim2n}.

\begin{remark}\label{remarkJoinAdjGH}
Let $G$ and $H$ be two graphs of order $n_1\ge 2$ and $n_2\ge 2$, respectively. Then $\adim_2(G+H)=n_1+n_2$ if and only if one of the two following statements hold.
\begin{enumerate}[{\rm(a)}]
\item Every vertex of $G$ belongs to a non-singleton twin equivalence class of $G$ and every vertex of $H$ belongs to a non-singleton twin equivalence class of $H$.
\item $\Delta(G)=n_1-1$, $\Delta(H)=n_2-1$, every vertex  $u\in V(G)$  of degree $\delta(u)<n_1-1$ belongs to a non-singleton twin equivalence class of $G$ and every vertex $v\in V(H)$  of degree $\delta(v)<n_2-1$  belongs to a non-singleton twin equivalence class of $H$.
\end{enumerate}
\end{remark}


Let $G$ and $H$ be two   graphs of order $n_1\ge 2$ and $n_2\ge 2$, respectively. If $x,y\in V(G)$, then $ \mathcal{C}_{G+H}(x,y)=\mathcal{C}_G(x,y)$. Analogously,  if $x,y\in V(H)$, then $ \mathcal{C}_{G+H}(x,y)=\mathcal{C}_H(x,y)$. Also, if $x\in V(G)$ and $y\in V(H)$, then
$ \mathcal{C}_{G+H}(x,y)=(V(G)-N_G(x))\cup (V(H)-N_H(y))$.
 Therefore,  $$\mathcal{C}(G+H)=\displaystyle\min\{\mathcal{C}(G),\mathcal{C}(H),n_1-\Delta(G)+n_2-\Delta(H)\}.$$

\begin{theorem}\label{propLowerBoundJoinG-H}
Let  $G$ and $H$ be two nontrivial graphs. Then the following assertions hold:
\begin{enumerate}[{\rm (i)}]
\item For any $k\in\{1,\ldots,\mathcal{C}(G+H)\}$, $$ \adim_k(G+H)\ge \adim_k(G)+\adim_k(H).$$
\item For any $k\in\{1,\ldots,\min\{\mathcal{C}(H), \mathcal{C}(K_1+G)\} \}$ $$\adim_k(G+H)\le \adim_k(K_1+G)+\adim_k(H).$$
\end{enumerate}
\end{theorem}

\begin{proof}
First we proceed  to deduce the lower bound. Let $A$ be a $k$-adjacency basis of $G+H$, $A_G=A\cap V(G)$, $A_H=A\cap V(H)$ and let $x,y\in V(G)$ be two different vertices. Notice that $A_G\ne \emptyset$ and $A_H\ne \emptyset$, as $n_1\ge 2$ and $n_2\ge 2$. Now, since $\mathcal{C}_{G+H}(x,y)=\mathcal{C}_{G}(x,y)$, it follows that $|A_G\cap\mathcal{C}_{G}(x,y)|=|A\cap\mathcal{C}_{G+H}(x,y)|\ge k$, and as a consequence, $A_G$ is a $k$-adjacency generator for $G$. By analogy  we deduce that $A_H$ is a $k$-adjacency generator for $H$. Therefore, $\adim_k(G+H)=|A|=|A_G|+|A_H|\ge\adim_k(G)+\adim_k(H)$.

To obtain the upper bound, first we suppose that there exists a $k$-adjacency basis $U$ of $K_1+G$ such that the vertex of $K_1$ does not belong to $U$. We claim that for any $k$-adjacency basis $B$ of $H$ the set   $X=U\cup B$ is a $k$-adjacency generator for $G+H$. To see this we take two different vertices $a,b\in V(G+H)$. If $a,b\in V(G)$, then $|\mathcal{C}_{G+H}(a,b)\cap X|=|\mathcal{C}_{K_1+G}(a,b)\cap U|\ge k$. If $a,b\in V(H)$, then $|\mathcal{C}_{G+H}(a,b)\cap X|=|\mathcal{C}_{H}(a,b)\cap B|\ge k$. Now, assume that $a\in V(G)$ and $b\in V(H)$. Since $U$ is a $k$-adjacency generator for $\langle b \rangle+G$, we have that $|\mathcal{C}_{\langle b \rangle+G}(a,b)\cap U|\ge k$. Hence, $|\mathcal{C}_{G+H}(a,b)\cap X| = |\mathcal{C}_{\langle b \rangle+G}(a,b)\cap U|\ge k$. Therefore, $X$ is a $k$-adjacency generator for $G+H$ and, as a consequence, $\adim_k(G+H)\le |X|= |U|+|B|= \adim_k(K_1+G)+\adim_k(H)$.

Suppose from now on that the vertex $u$ of $K_1$ belongs to any $k$-adjacency basis $U$ of $K_1+G$. We differentiate two cases:

Case 1. For any $k$-adjacency basis $B$ of $H$, there exists a vertex $x$ such that $B \subseteq N_H(x)$. We claim that $X=U'\cup (B\cup\{x\})$ is a $k$-adjacency generator for $G+H$, where $U'=U-\{u\}$. To see this we take two different vertices $a,b\in V(G+H)$. Notice that since $B$ is $k$-adjacency basis of $H$, there exists exactly one vertex $x\in V(H)$ such that $B\subseteq N_H(x)$ and for any $y\in V(H)-\{x\}$ it holds $|B-N_H(y)|\ge k$. If $a,b\in V(G)$, then $|\mathcal{C}_{G+H}(a,b)\cap X|=|\mathcal{C}_{K_1+G}(a,b)\cap U'|=|\mathcal{C}_{K_1+G}(a,b)\cap U|\ge k$. If $a,b\in V(H)$, then $|\mathcal{C}_{G+H}(a,b)\cap X|=|\mathcal{C}_{H}(a,b)\cap (B\cup\{x\})|\ge k$. Now, assume that $a\in V(G)$ and $b\in V(H)$. Since $U'\cup \{b\}$ is a $k$-adjacency basis of $\langle b \rangle+G$, we have that $|\mathcal{C}_{\langle b \rangle+G}(a,b)\cap U'|\ge k-1$. Furthermore, $|\mathcal{C}_{\langle a \rangle+H}(a,b)\cap (B\cup\{x\})|\ge 1$. Hence, $|\mathcal{C}_{G+H}(a,b)\cap X| = |\mathcal{C}_{\langle b \rangle+G}(a,b)\cap U'|+|\mathcal{C}_{\langle a \rangle+H}(a,b)\cap (B\cup\{x\})|\ge k$. Therefore, $X$ is a $k$-adjacency generator for $G+H$ and, as a consequence, $\adim_k(G+H)\le |X|= |U'|+|B\cup\{x\}|= (\adim_k(K_1+G)-1)+(\adim_k(H)+1)=\adim_k(K_1+G)+\adim_k(H)$.

Case 2. There exists a $k$-adjacency basis $B'$ of $H$ such that $|B'- N_H(h')|\ge 1$, for all $h'\in V(H)$. We take $X=U'\cup B'$ and we proceed as above to show  that $X$ is a $k$-adjacency generator for $G+H$. As above, for $a,b\in V(G)$ or $a,b\in V(H)$ we deduce that $|\mathcal{C}_{G+H}(a,b)\cap X|\ge k$. Now, for $a\in V(G)$ and $b\in V(H)$ we have $|\mathcal{C}_{\langle b \rangle+G}(a,b)\cap U'|\ge k-1$ and $|\mathcal{C}_{\langle a \rangle+H}(a,b)\cap B'|\ge 1$. Hence, $|\mathcal{C}_{G+H}(a,b)\cap X| = |\mathcal{C}_{\langle b \rangle+G}(a,b)\cap U'|+|\mathcal{C}_{\langle a \rangle+H}(a,b)\cap B|\ge k$ and, as a consequence, $\adim_k(G+H)\le |X|= |U'|+|B'|= (\adim_k(K_1+G)-1)+\adim_k(H)\le\adim_k(K_1+G)+\adim_k(H)$.
\end{proof}

By Proposition \ref{ResultOmoomiDimK1+H} and Theorem \ref{propLowerBoundJoinG-H}   we obtain the following result. 

\begin{proposition}
Let  $G$ and $H$ be two  non-trivial  graphs. If for any adjacency basis $A$ of $G$, there exists $g\in V(G)$ such that $A\subseteq  N_G(g)$ and for any adjacency basis $B$ of $H$, there exists $h\in V(H)$ such that $B\subseteq  N_H(h)$,  then  $$\adim_1(G+H)=\adim_1(G)+\adim_1(H)+1$$ Otherwise,
$$\adim_1(G+H)=\adim_1(G)+\adim_1(H).$$ 
\end{proposition}

\begin{corollary}\label{FromAdimK1+HtoAdimG+H}
Let  $G$ and $H$ be two nontrivial graphs and $k\in\{1,\ldots,\mathcal{C}(G+H)\}$. If $\adim_k(K_1+G)=\adim_k(G)$, then  $$\adim_k(G+H)=\adim_k(G)+\adim_k(H).$$
\end{corollary}

In the previous section we showed that there are several classes of graphs where $\adim_k(K_1+G)=\adim_k(G)$. This is the case, for instance, of graphs of diameter $D(G) \ge 6$, or  $G\in \{P_n,C_n\}$, $n\ge 7$, or  graphs  of girth $\mathtt{g}(G)\ge 5$ and minimum degree $\delta(G)\ge 3$.  Hence, for any of these graphs, any nontrivial graph $H$, and any $k\in\{1,\ldots,\min\{\mathcal{C}(H), \mathcal{C}(K_1+G)\} \}$  we have that $\adim_k(G+H)=\adim_k(G)+\adim_k(H).$

\begin{theorem}\label{propEqualJoinG-H}
Let $G$ and $H$ be two nontrivial graphs.  Then the following assertions are equivalent: 

\begin{enumerate}[{\rm (i)}]
\item There exists a $k$-adjacency basis $A_G$ of $G$ and  a $k$-adjacency basis $A_H$ of $H$ such that $|(A_G-N_G(x))\cup (A_H-N_H(y))|\ge k$, for all  $x\in V(G)$ and $y\in V(H)$.
\item $\adim_k(G+H)=\adim_k(G)+\adim_k(H).$
\end{enumerate}
\end{theorem}

\begin{proof}
Let $A_G$ be a $k$-adjacency basis  of $G$ and  and let $A_H$ be a $k$-adjacency basis  of $H$ such that $|(A_G-N_G(x))\cup (A_H-N_H(y))|\ge k$, for all  $x\in V(G)$ and $y\in V(H)$.   By Theorem \ref{propLowerBoundJoinG-H},  $\adim_k(G+H)\ge\adim_k(G)+\adim_k(H)$. It remains to prove that $\adim_k(G+H)\le\adim_k(G)+\adim_k(H)$.  We will prove that $A=A_G\cup A_H$ is a $k$-adjacency generator for $G+H$. We differentiate three cases for two vertices $x,y\in V(G+H)$. If $x,y\in V(G)$, then the fact that $A_G$ is a $k$-adjacency basis of $G$ leads to $k\le|A_G\cap\mathcal{C}_{G}(x,y)|=|A\cap\mathcal{C}_{G+H}(x,y)|$. Analogously we deduce the case $x,y\in V(H)$. If $x\in V(G)$ and $y\in V(H)$, then the fact that $\mathcal{C}_{G+H}(x,y)=(V(G)-N_G(x))\cup(V(H)-N_H(y))$ and $|(A_G-N_G(x))\cup(A_H-N_H(y))|\ge k$ leads to $|A\cap\mathcal{C}_{G+H}(x,y)|\ge k$. Therefore, $A$ is a $k$-adjacency generator for $G+H$, as a consequence, $|A|=|A_G|+|A_H|=\adim_k(G)+\adim_k(H)\ge\adim_k(G+H)$. 

 On the other hand, let $B$ be a $k$-adjacency basis of $G+H$ such that $|B|=\adim_k(G)+\adim_k(H)$ and let $B_G=B\cap V(G)$ and $B_H=B\cap V(H)$. Since for any $g_1,g_2\in V(G)$ and $h\in V(H)$, $h\not\in \mathcal{C}_{G+H}(g_1,g_2)$, we conclude that $B_G$ is a $k$-adjacency generator for $G$ and, by analogy, $B_H$ is a $k$-adjacency generator for $H$. Thus, $|B_G|\ge \adim_k(G)$, $|B_H|\ge \adim_k(H)$ and  $|B_G|+|B_H|=|B|=\adim_k(G)+ \adim_k(H)$. Hence, $|B_G|=\adim_k(G)$, $|B_H|=\adim_k(H)$ and, as a consequence, $B_G$ and $B_H$ are $k$-adjacency bases of $G$ and $H$, respectively. If there exists $g\in V(G)$ and $h\in V(H)$ such that $|(B_G-N_G(g))\cup (B_H-N_H(h))|< k$, then  
$|B\cap \mathcal{C}_{G+H}(g,h)|=|(B_G-N_G(g))\cup (B_H-N_H(h))|<k$, which is a contradiction. Therefore, the result follows.
\end{proof}

We would point out the following particular cases of the previous result\footnote{Notice that for $n\ge 7$ and $n'\ge 6$, this result can be derived from Corollary \ref{FromAdimK1+HtoAdimG+H}.}.

\begin{corollary}
Let $C_n$ be a cycle graph of order $n\ge 5$ and $P_{n'}$ a path graph of order $n'\ge 4$.
If $G\in \{K_t+C_n,N_t+C_n\}$, then
$$\adim_1(G)=\left\lfloor\frac{2n+2}{5}\right\rfloor+t-1\; and \; \adim_2(G)=\left\lceil\frac{n}{2}\right\rceil+t. $$
 If $G\in \{K_t+P_{n'},N_t+P_{n'}\}$, then 
$$\adim_1(G)=\left\lfloor\frac{2n'+2}{5}\right\rfloor+t-1 \; and \;\adim_2(G)=\left\lceil\frac{n'+1}{2}\right\rceil+t.$$
\end{corollary}

\begin{proof}
Let $G_1\in \{K_t,N_t\}$ and $G_2\in \{P_n,C_n\}$. By  Propositions  \ref{value-adj-Paths} and \ref{value-adj-Cycles} we deduce that  $\adim_2(G_2)-\Delta(G_2)\ge 1$. 
On the other hand, for any $2$-adjacency basis $A$ of $G_1$ and $x\in V(G_1)$ we have $|B-N_{G_1}(y)|\in \{1,t\}$. 
Therefore, by Theorem \ref{propEqualJoinG-H} we obtain the result for $G=G_1+G_2$. 
\end{proof}

\begin{corollary}
Let $G$ be a graph of order $n\ge 7$ and maximum degree $\Delta(G)\le 3$. Then for any integer $t\ge 2$ and $H\in \{K_t,N_t\}$,
$$\adim_2(G+H)=\adim_2(G)+t.$$
\end{corollary}
\begin{proof}
By Theorem \ref{Bound-le-k+2} we deduce that $\adim_2(G)\ge 4$, so 
for any $2$-adjacency basis $A$ of $G$ and $x\in V(G)$ we have $|A-N_G(x)|\ge 1$. Moreover, for any $2$-adjacency basis $B$ of $H$ and $y\in V(H)$ we have $|B-N_{H}(y)|\in \{1,t\}$. 
Therefore, by Theorem \ref{propEqualJoinG-H} we obtain the result. 
\end{proof}

\begin{corollary}\label{CorollarySumadimGradoPeque}
Let $G$ and $H$ be two graphs of order at least seven such that $G$ is $k_1$-adjacency dimensional and $H$ is $k_2$-adjacency dimensional. For any integer $k$ such that $\Delta(G)+\Delta(H)-4\le k\le \min\{k_1,k_2\}$,
$$\adim_k(G+H)=\adim_k(G)+\adim_k(H).$$
\end{corollary}

\begin{proof}
By Theorem \ref{Bound-le-k+2}, for any positive integer $k\le \min\{k_1,k_2\}$, we have $\adim_k(G)\ge k+2$ and $\adim_k(H)\ge k+2$.
Thus, if $k \ge \Delta(G)+\Delta(H)-4$, then
$(\adim_k(G)-\Delta(G))+(\adim_k(H)-\Delta(H))\ge k$. Therefore, by Theorem \ref{propEqualJoinG-H} we conclude the proof.
\end{proof}

As a particular case of the result above we derive the following remark.

\begin{remark}
Let $G$ and $H$ be two $3$-regular graphs of order at least seven. Then $$\adim_2(G+H)=\adim_2(G)+\adim_2(H).$$
\end{remark}

\end{document}